\documentclass[11pt]{amsart}
\usepackage{pdfsync}
\usepackage{latexsym}
\usepackage{amsgen,amsmath,amsxtra,amssymb,amsfonts,amsthm}
\usepackage{times}
\usepackage{graphicx}
\usepackage[dvips]{epsfig}
\usepackage{subfigure}
\usepackage[active]{srcltx}
\usepackage{color}
\usepackage{a4,amsthm,amsfonts,amsmath}

\begin{document}
\newcommand{\dyle}{\displaystyle}
\newcommand{\Hi}{{\mathbb H}}
\newcommand{\Ss}{{\mathbb S}}
\newcommand{\Rn}{{\mathbb{R}^n}}
\newcommand{\ieq}{\begin{equation}}
\newcommand{\eeq}{\end{equation}}
\newcommand{\ieqa}{\begin{eqnarray}}
\newcommand{\eeqa}{\end{eqnarray}}
\newcommand{\ieqas}{\begin{eqnarray*}}
\newcommand{\eeqas}{\end{eqnarray*}}
\newcommand{\Bo}{\put(260,0){\rule{2mm}{2mm}}\\}
\newcommand{\RR}{{\mathord{\mathbb R}}}
\newcommand{\Z}{{\mathord{\mathbb Z}}}
\newcommand{\N}{{\mathord{\mathbb N}}}
\newcommand{\C}{{\mathord{\mathbb C}}}
\def\L#1{\label{#1}} \def\R#1{{\rm (\ref{#1})}}


\theoremstyle{plain}
\newtheorem{theorem}{Theorem} [section]
\newtheorem{corollary}[theorem]{Corollary}
\newtheorem{lemma}[theorem]{Lemma}
\newtheorem{proposition}[theorem]{Proposition}
\newtheorem{co}{Corollary}
\newtheorem{lm}{Lemma}
\newtheorem{q}{Theorem}
\def\neweq#1{\begin{equation}\label{#1}}
\def\endeq{\end{equation}}
\def\eq#1{(\ref{#1})}

\theoremstyle{definition}
\newtheorem{definition}[theorem]{Definition}
\newtheorem{remark}[theorem]{Remark}

\numberwithin{figure}{section}
\newcommand{\res}{\mathop{\hbox{\vrule height 7pt width .5pt depth
0pt \vrule height .5pt width 6pt depth 0pt}}\nolimits}
\def\at#1{{\bf #1}: } \def\att#1#2{{\bf #1}, {\bf #2}: }
\def\attt#1#2#3{{\bf #1}, {\bf #2}, {\bf #3}: } \def\atttt#1#2#3#4{{\bf #1}, {\bf #2}, {\bf #3},{\bf #4}: }
\def\aug#1#2{\frac{\displaystyle #1}{\displaystyle #2}} \def\figura#1#2{ \begin{figure}[ht] \vspace{#1} \caption{#2}
\end{figure}} \def\B#1{\bibitem{#1}} \def\q{\int_{\Omega^\sharp}}
\def\z{\int_{B_{\bar{\rho}}}\underline{\nu}\nabla (w+K_{c})\cdot
\nabla h} \def\a{\int_{B_{\bar{\rho}}}}
\def\b{\cdot\aug{x}{\|x\|}}
\def\n{\underline{\nu}} \def\d{\int_{B_{r}}}
\def\e{\int_{B_{\rho_{j}}}} \def\LL{{\mathcal L}}
\def\itr{\mathrm{Int}\,}
\def\D{{\mathcal D}}
\def\tg{\tilde{g}}
\def\A{{\mathcal A}}
\def\S{{\mathcal S}}
\def\H{{\mathcal H}}
\def\M{{\mathcal M}}
\def\T{{\mathcal T}}
\def\U{{\mathcal U}}
\def\N{{\mathcal N}}
\def\I{{\mathcal I}}
\def\F{{\mathcal F}}
\def\J{{\mathcal J}}
\def\E{{\mathcal E}}
\def\F{{\mathcal F}}
\def\G{{\mathcal G}}
\def\HH{{\mathcal H}}
\def\W{{\mathcal W}}
\def\H{\D^{2*}_{X}}
\def\d{d^X_M }
\def\LL{{\mathcal L}}
\def\H{{\mathcal H}}
\def\HH{{\mathcal H}}
\def\itr{\mathrm{Int}\,}
\def\vah{\mbox{var}_\Hi}
\def\vahh{\mbox{var}_\Hi^1}
\def\vax{\mbox{var}_X^1}
\def\va{\mbox{var}}
\def\SS{{\mathcal S}}
\def\Y{{\mathcal Y}}
\def\length{{l_\Hi}}
\newcommand{\average}{{\mathchoice {\kern1ex\vcenter{\hrule
height.4pt width 6pt depth0pt} \kern-11pt} {\kern1ex\vcenter{\hrule height.4pt width 4.3pt depth0pt} \kern-7pt} {} {} }}
\def\weak{\rightharpoonup}
\def\detu{{\rm det}(D^2u)}
\def\detut{{\rm det}(D^2u(t))}
\def\detvt{{\rm det}(D^2v(t))}
\def\detv{{\rm det}(D^2v)}
\def\uuu{u_xu_yu_{xy}}
\def\uuut{u_x(t)u_y(t)u_{xy}(t)}
\def\uuus{u_x(s)u_y(s)u_{xy}(s)}
\def\uuutn{u_x(t_n)u_y(t_n)u_{xy}(t_n)}
\def\vvv{v_xv_yv_{xy}}
\newcommand{\ave}{\average\int}

\title[Polyharmonic $k-$Hessian equations]{Polyharmonic $k-$Hessian equations in $\mathbb{R}^N$}

\author[P. Balodis, C. Escudero]{Pedro Balodis, Carlos Escudero}
\address{}
\email{}

\date{\today}

\begin{abstract}
This work is focused on the study of the nonlinear elliptic higher order equation
\begin{equation}\nonumber
\left( -\Delta \right)^m u = S_k[-u] + \lambda f, \qquad x \in \mathbb{R}^N,
\end{equation}
where the $k-$Hessian $S_k[u]$ is the $k^{\mathrm{th}}$ elementary symmetric polynomial of eigenvalues of the Hessian matrix of the solution
and the datum $f$ belongs to a suitable functional space. This problem is posed in $\mathbb{R}^N$ and we prove the existence
of at least one solution by means of topological fixed point methods for suitable values of $m \in \mathbb{N}$. Questions related to
the regularity of the solutions and extensions of these results to the nonlocal setting are also addressed.
On the way to construct these proofs, some technical results such as a fixed point theorem and a refinement of the critical Sobolev embedding,
which could be of independent interest, are introduced.
\end{abstract}
\renewcommand{\thefootnote}{\fnsymbol{footnote}}
\setcounter{footnote}{-1}
\footnote{Supported by project MTM2013-40846-P, MINECO, Spain.}
\renewcommand{\thefootnote}{\arabic{footnote}}
\maketitle

\section{Introduction}

The goal of this work is to develop an analytical framework for the study of the family of higher order equations
\begin{equation}\label{rkhessian}
\left( -\Delta \right)^m u = S_k[-u] + \lambda f, \qquad x \in \mathbb{R}^N,
\end{equation}
where $m, \, N, \, k \, \in \mathbb{N}$, $\lambda \in \mathbb{R}$ and the datum $f:\mathbb{R}^N \longrightarrow \mathbb{R}$ belongs to a
suitable functional space, to be made precise in the following. The nonlinearity in this equation is the $k-$Hessian $S_k[u]=\sigma_k(\Lambda)$, where
$$
\sigma_k(\Lambda)= \sum_{i_1<\cdots<i_k} \Lambda_{i_1} \cdots \Lambda_{i_k},
$$
is the $k^{\mathrm{th}}$ elementary symmetric polynomial and $\Lambda=(\Lambda_1,\cdots,\Lambda_n)$ are the eigenvalues of the Hessian matrix of
the solution $(D^2 u)$.
Analogously $S_k[u]$ can be defined as the sum of the $k^{\mathrm{th}}$ principal minors of the Hessian matrix or,
using the language of exterior algebra, as the trace of the $k^{\mathrm{th}}$ exterior power of $(D^2 u)$.
For $k=1$ the $k-$Hessian $S_k[u]$ becomes the trace of the Hessian matrix, that is, the Laplacian. Since our focus is put on nonlinear equations
we will skip this case and always consider $2 \le k \le N$.

To describe our motivation consider for a moment equation~\eqref{rkhessian} free of the polyharmonic operator.
Such an equation would not only generalize the Poisson equation for $k=1$, it would also generalize the Monge-Amp\`ere
equation~\cite{caffarelli1,caffarelli2}
$$
\det(D^2 u) = f,
$$
for $k=N$. In fact, such an equation
$$
S_k[u] = f,
$$
is denominated the $k-$Hessian equation, and it, together with related problems, has been intensively studied during the last
years~\cite{caffarelli3,wang1,wang2,labutin,wang3,trudinger,trudinger1,wang4,wang5,wang6,wang7,wang8,wang9,wang10,wang11,wang12,wang}.
It is interesting to note that the analytical approach to this problem has required the assumption of a series of geometric constraints
in order to preserve the ellipticity of the nonlinear $k-$Hessian operator~\cite{wang}. Such constraints are not needed in the case of
full equation~\eqref{rkhessian}~\cite{escudero}, what makes this sort of problem an alternative viewpoint to the interesting nonlinear
$k-$Hessian operator.

A second source of motivation is the rise of studies focused on polyharmonic problems in recent times~\cite{AGGM,BG,CEGM,DDGM,DFG,FG,FGK,GGS,moradifam}.
While boundary value problems for polyharmonic operators have already been considered with different types of interesting nonlinearities in these
and different works, the history of polyharmonic $k-$Hessian equations is still short~\cite{escudero,n1,n2,n3,n4,n5,n6}.
At this point, it is important to stress the natural character of this sort of nonlinearity in the polyharmonic framework. Indeed,
the $k-$Hessians, $1 \le k \le N$, form a basis of the vector space of polynomial invariants of the Hessian matrix under the orthogonal group $O(N)$
of degree lower or equal to $N$, at least for regular enough $u$~\cite{procesi}. So on one hand these nonlinearities give rise to genuinely
polyharmonic semilinear equations with no possible harmonic analogue, what makes them an excellent candidate to push forward the theory of polyharmonic
boundary value problems. While on the other hand, these higher order equations are some of the simplest ones compatible with the ideas of invariance
with respect to rotations and reflections widespread in the realm of physical modeling.

Yet another interesting property that motivates us to study equation~\eqref{rkhessian} is its intriguing dependence on the boundary conditions,
as already noted in~\cite{n5}. We studied in~\cite{escudero} this family of equations on bounded domains subject to Dirichlet boundary conditions.
In this work we are interested on the ``boundary value problem''
\begin{subequations}
\begin{eqnarray}\label{bvphessian}
\left( -\Delta \right)^m u &=& S_k[-u] + \lambda f, \qquad x \in \mathbb{R}^N, \\ \label{bcons}
u(x) &\rightarrow& 0, \quad \text{when} \quad |x| \rightarrow \infty.
\end{eqnarray}
\end{subequations}
First of all we have to state what do we mean by this ``boundary condition''; in fact, this constitutes a very important \textit{remark}:
we say that a solution ``vanishes at infinity'' if it belongs to some $L^p(\mathbb{R}^N)$, $1 \le p < \infty$, although we cannot give any
reasonable pointwise meaning to such an affirmation. Note that this is the only way in which an existence theory \textit{\`a la}
Calder\'on-Zygmund can be pushed forward. Of course, if a function pointwise vanishes at infinity, we will also say that it ``vanishes at infinity''.
Note also that the nonlinearity is $S_k[-u]$ rather than $S_k[u]$; that is, the nonlinearity is exactly the coefficient of the
monomial of degree $N-k$ within the characteristic polynomial of the Hessian matrix. We have considered such a form to be
in complete agreement with the structure of the equation in~\cite{escudero}. However, this assumption was needed in this reference
in order to construct the variational approach to the existence of solutions employed there. Our present approach relies on a
topological fixed point argument and would work exactly in the same way if we substituted the current nonlinearity by $S_k[u]$.
This, among other things, highlights the fact that the present existence proofs are genuinely different from previously used arguments.

We now present our main result:
\begin{theorem}\label{main}
Problem~\eqref{bvphessian}-\eqref{bcons} has at least one weak solution in the following cases:
\begin{itemize}
\item[(a)] $f \in L^p(\mathbb{R}^N), \quad 1 < p < \frac{N}{2k}, \quad m = 1 + N (k-1)/(2pk) \in \mathbb{N}, \quad N > 2k$,
\item[(b)] $f \in L^1(\mathbb{R}^N), \quad m=1 + N (k-1)/(2k) \in \mathbb{N}, \quad N > 2k$,
\item[(c)] $f \in \mathcal{H}^1(\mathbb{R}^N), \quad m=1 + N (k-1)/(2k) \in \mathbb{N}, \quad N > 2k$,
\item[(d)] $f \in \mathcal{H}^1(\mathbb{R}^N), \quad m=1 + N (k-1)/(2k) \in \mathbb{N}, \quad N = 2k$,
\end{itemize}
provided $|\lambda|$ is small enough.
Then, respectively
\begin{itemize}
\item[(a)] $u \in \dot{W}^{2m-\epsilon,N p/(N-\epsilon p)}(\mathbb{R}^N) \, \forall \, 0 \le \epsilon \le 2m$,
\item[(b)] $u \in \dot{W}^{2m-\epsilon,N/(N-\epsilon)}(\mathbb{R}^N) \, \forall \, 0<\epsilon \le 2m$,
\item[(c)] $u \in \dot{W}^{2m-\epsilon,N/(N-\epsilon)}(\mathbb{R}^N) \, \forall \, 0 \le \epsilon \le 2m$,
\item[(d)] $u \in \dot{W}^{2m-\epsilon,N/(N-\epsilon)}(\mathbb{R}^N) \, \forall \, 0 \le \epsilon \le 2m$.
\end{itemize}
Moreover, in case (b), $D^{2m} u \in L^{1,\infty}(\mathbb{R}^N)$, in case (c), $D^{2m} u \in \mathcal{H}^{1}(\mathbb{R}^N)$
and, in case (d), $D^{2m} u \in \mathcal{H}^{1}(\mathbb{R}^N)$ and $u \in C_0(\mathbb{R}^N)$.
Also, for a smaller enough $|\lambda|$, the solution is locally unique in cases (a), (b) and (c).
\end{theorem}

\begin{proof}
The statement follows as a consequence of Theorems~\ref{exh1}, \ref{exl1}, \ref{exlp}, \ref{exul1}, \ref{exulp}
and Corollary~\ref{corc0}.
\end{proof}

\begin{remark}
Note that, in case (d), $m=k$ always, so problem~\eqref{bvphessian}-\eqref{bcons} reduces to
\begin{subequations}
\begin{eqnarray}\nonumber
\left( -\Delta \right)^k u &=& S_k[-u] + \lambda f, \qquad x \in \mathbb{R}^{2k}, \\ \nonumber
u(x) &\rightarrow& 0, \quad \text{when} \quad |x| \rightarrow \infty,
\end{eqnarray}
\end{subequations}
for any $k \ge 2$.
\end{remark}

\begin{remark}
It is important to note that our methods are applicable to more general families of nonlinearities. Denote by $R_k^j(\cdot)$ the $j-$th
principal minor of order $k$. The present results hold as well if we substituted $S_k(-u)$ by $R_k^j(-u)$ in equation~\eqref{bvphessian} for any $j$.
In fact, the nonlinearities $S_k(-u)$ are just a particular linear combination of these $R_k^j(-u)$; and our theory could be constructed actually for
\emph{any} linear combination of them.
This comes from the fact that we need two main ingredients in our proofs: weak continuity of the maps $S_k$ and the fact that they also preserve the
$L^p$ and Hardy spaces the datum $f$ belongs to. The same holds, for example,
for the maps $R_k^j$, see~\cite{grafakos,grafakos1}, and for any linear combination
of them by linearity. Our main attention lies, however, in the operators $S_k$ described before due to their simple geometric meaning which is at least
not as evident for the operators $R_k^j$ or their arbitrary linear combinations.
\end{remark}

Now we describe the remainder of the article. In section~\ref{funfra} we introduce the functional framework we need in our proofs and some
notation. In section~\ref{ltheory} we developed the theory that corresponds to the linear counterpart of problem~\eqref{bvphessian}-\eqref{bcons}.
In section~\ref{topofpt} we state and prove a topological fixed point theorem that will be the main abstract tool for proving existence of
solutions to our differential problem. In section~\ref{refsobolev} we prove a refinement of the classical critical Sobolev embedding that will
be subsequently needed in the following section. These last two sections could be of independent interest and, as such, they have been written in
a self-contained fashion. Our main existence results come in section~\ref{existence}, and the local uniqueness results in section~\ref{locuniq}.
A nonlocal extension of Theorem~\ref{main} is proven in section~\ref{nproblems} and, finally, some further results regarding the weak continuity
of the branch of solutions and some extra regularity for the critical case (d) are described in section~\ref{further}.

\section{Functional Framework and Notation}
\label{funfra}

In order to build the existence theory for our partial differential equation
we need to introduce the Hardy space $\mathcal{H}^1$ in $\mathbb{R}^N$~\cite{stein} and its dual,
the space of functions of bounded mean oscillation.

\begin{definition}\label{hardyspace} Let $\Phi \in \mathcal{S}(\mathbb{R}^N)$, where $\mathcal{S}(\mathbb{R}^N)$ denotes
the Schwartz space, be a function such that $\int_{\mathbb{R}^N} \Phi \, dx =1$.
Define $\Phi_s := s^{-N} \Phi(x/s)$ for $s>0$. A locally integrable function $f$ is said to be in $\mathcal{H}^1(\mathbb{R}^N)$ if
the maximal function
$$
\mathcal{M}f(x):= \sup_{s>0} \left| \Phi_s \ast f(x) \right|
$$
belongs to $L^1(\mathbb{R}^N)$. We define the norm $\|f\|_{\mathcal{H}^1(\mathbb{R}^N)}=\|\mathcal{M} f\|_1$.
\end{definition}

\begin{remark}
There are several equivalent definitions of this space, see~\cite{steinb}.
\end{remark}

Now we introduce the space of functions of bounded mean oscillation~\cite{steinb}.

\begin{definition}
A locally integrable function $f$ is said to be in $\mathrm{BMO}(\mathbb{R}^N)$ if the seminorm (or norm in the quotient space
of locally integrable functions modulo additive constants)
$$
\|f\|_{\mathrm{BMO}(\mathbb{R}^N)}:= \sup_Q \frac{1}{|Q|}\int_Q |f(x)-f_Q| \, dx,
$$
where $|Q|$ is the Lebesgue measure of $Q$, $f_Q=\frac{1}{|Q|}\int_Q f(x) \, dx$ and the supremum is taken
over the set of all cubes $Q \subset \mathbb{R}^N$, is finite.
\end{definition}

We also need the pre-dual of the Hardy space $\mathcal{H}^1(\mathbb{R}^N)$.

\begin{definition}
We define $\mathrm{VMO}(\mathbb{R}^N)$ as the closure of $C_0(\mathbb{R}^N)$ in $\mathrm{BMO}(\mathbb{R}^N)$,
with $\| f \|_{\mathrm{VMO}(\mathbb{R}^N)}=\| f \|_{\mathrm{BMO}(\mathbb{R}^N)} \, \forall \, f \in \mathrm{VMO}(\mathbb{R}^N)$.
\end{definition}

The following functional spaces will also be useful in the construction of the existence theory.

\begin{definition}
We define the homogeneous Sobolev space $\dot{W}^{j,p}(\mathbb{R}^N)$ as the space of all measurable functions $u$ that are $j$ times weakly derivable and whose weak derivatives of $j-$th order obey
$$
\|D^j u\|_p < \infty,
$$
where $\|\cdot\|_p$ denotes the norm of $L^p(\mathbb{R}^N)$, $1 \le p \le \infty$, $j \in \mathbb{N}$.
\end{definition}

In our derivations we will need the following operators.

\begin{definition}
We define the Riesz transforms in $\mathbb{R}^N$:
$$
R_{x_j}(f)(x)= \frac{\Gamma \left(\frac{n+1}{2}\right)}{\pi^{(n+1)/2}} \,\, \text{P. V.} \int_{\mathbb{R}^N} \frac{x_j-y_j}{|x-y|^{n+1}} \, f(y) \, dy.
$$
\end{definition}

\begin{remark}
The normalization of the Riesz transforms is chosen in such a way that
$$
\mathcal{F}[R_{x_j}(f)](\xi)= \pi i \, \frac{\xi_j}{|\xi|} \, \mathcal{F}(f)(\xi).
$$
\end{remark}

Finally, we introduce two definitions relating to real numbers and their relationships.

\begin{definition}
Let $x_\alpha, y_\alpha \in \RR$ ($\alpha\in A$, $A$ some set).
We write $x \ll y$ ($x=\{x_\alpha\}_{\alpha \in A}$, $y=\{y_\alpha\}_{\alpha \in A}$)
whenever there exists a positive constant $c$ such that $x_\alpha \le c y_\alpha$ for every $\alpha\in A$.
\end{definition}

\begin{definition}
We denote $\mathbb{R}_+:=\{ x \in \mathbb{R} | x \ge 0 \}$.
\end{definition}

\section{Linear Theory}
\label{ltheory}

This section is devoted to the study of the linear problem
\begin{equation}\label{linear}
\left( -\Delta \right)^m u = \lambda f, \qquad x \in \mathbb{R}^N,
\end{equation}
where $m \in \mathbb{N}$ and we consider the ``boundary condition'' $u \to 0$ when $|x| \to \infty$.

\begin{proposition}\label{exun}
Equation~\eqref{linear} has a unique solution in the following cases:
\begin{itemize}
\item[(a)] $f \in L^p(\mathbb{R}^N), \quad 1 < p < \frac{N}{2m}, \quad m < N/2$,
\item[(b)] $f \in L^1(\mathbb{R}^N), \quad m < N/2$,
\item[(c)] $f \in \mathcal{H}^1(\mathbb{R}^N), \quad m < N/2$,
\item[(d)] $f \in \mathcal{H}^1(\mathbb{R}^N), \quad m = N/2$.
\end{itemize}
Then, respectively
\begin{itemize}
\item[(a)] $u \in L^q(\mathbb{R}^N) \cap \dot{W}^{2m,p}(\mathbb{R}^N)$,
\item[(b)] $u \in \dot{W}^{2m-\epsilon,N/(N-\epsilon)}(\mathbb{R}^N) \, \forall \, 0<\epsilon \le 2m$,
\item[(c)] $u \in L^{q'}(\mathbb{R}^N) \cap \dot{W}^{2m,1}(\mathbb{R}^N)$,
\item[(d)] $u \in L^\infty(\mathbb{R}^N) \cap \dot{W}^{2m,1}(\mathbb{R}^N)$,
\end{itemize}
where $q=Np/(N-2mp)$ and $q'=N/(N-2m)$. Moreover, in case (b), $D^{2m} u \in L^{1,\infty}(\mathbb{R}^N)$,
and, in all cases, the map $f \mapsto u$ is continuous.
\end{proposition}

\begin{proof}

The proof focuses on the range $m \ge 2$ since the case $m=1$ is classical.

{\sc Step 1.}

We start considering the auxiliary problem
\begin{equation}\label{deltasol}
\left( -\Delta \right)^m G = \delta_0, \qquad x \in \mathbb{R}^N,
\end{equation}
where $\delta_0$ is the unit Dirac mass centered at the origin. The explicit solution to this equation is well known~\cite{GGS}:
\begin{equation}\label{funsol}
G(x)= \left\{
\begin{array}{cc}
\frac{-\log|x|}{N V_N 4^{m-1} \Gamma(N/2) (m-1)!} & \text{if } N=2m, \\ \\
\frac{2 \Gamma(N/2-m)}{N V_N 4^m \Gamma(N/2)(m-1)!} \frac{1}{|x|^{N-2m}} & \text{in other case},
\end{array}
\right.
\end{equation}
where $V_N= \pi^{N/2}/\Gamma(1+N/2)$ is the volume of the $N-$dimensional unit ball, and always under the assumption $N \ge 2m$.

The unique solution to equation~\eqref{linear} is given by the convolution
\begin{equation}\label{convolution}
u= \lambda \, G * f.
\end{equation}
Now we justify that this is a well defined function in a suitable functional space.

{\sc Step 2.}

For $N > 2m$ we have $G \propto |x|^{2m-N}$, therefore $G$ defines a Newtonian potential
$$
I_{2m}(f)= \int_{\mathbb{R}^N} G(x-y) \, f(y) \, dy,
$$
and, as such, $\|I_{2m}(f)\|_q \ll \| f \|_p $, see~\cite{grafakos2}, and therefore
$$
\|u\|_q \ll |\lambda| \, \|f\|_p,
$$
where $q=Np/(N-2mp)$, in case (a). Cases (b) and (c) follow analogously.

For $N = 2m$ we have $G \propto \log|x|$ and since in this case $f \in \mathcal{H}^1(\mathbb{R}^N)$, and $\log|x| \in \text{BMO}(\mathbb{R}^N)$, it follows that
$$
\| u \|_\infty \ll |\lambda| \, \|f\|_{\mathcal{H}^1(\mathbb{R}^N)}.
$$

{\sc Step 3.}

For the regularity of $u$ it suffices to show that $D^{2m}G$ defines a singular integral operator~\cite{grafakos2}. Note that
$$
\Delta |x|^{-\alpha}= \frac{(\alpha+2-N)\alpha}{|x|^{\alpha+2}} \quad \forall \, \alpha>0.
$$
If we denote $C_{\alpha,N}:=\alpha(\alpha+2-N)$ and $K_{N,m}:= G(x) |x|^{N-2m}$ whenever $N > 2m$, we have
$$
\left( -\Delta \right)^{m-1} G(x) = (-1)^m K_{N,m} C_{N-2m,N} C_{N-2(m-1),N} \cdots C_{N-4,N} |x|^{2-N}.
$$
On the other hand, it is easy to check that
$$
\partial_{x_j x_k}^2 |x|^{2-N} = (2-N) \frac{|x|^2 \delta_{jk} - N x_j x_k}{|x|^{N+2}}.
$$
Note also the average of the numerator over the unit sphere
$$
I_{jk}= \int_{S^{N-1}} (|x|^2 \delta_{jk} - N x_j x_k) dw = \delta_{jk} |S^{N-1}|-N \int_{S^{N-1}} w_j w_k dw=0.
$$
We denote $\partial_{jk}^2 := \partial_{x_j x_k}^2$ and define the operator
$$
T_{j,k}(f):=\partial_{jk}^2(-\Delta)^{m-1}u,
$$
which is clearly a singular integral operator in $\mathbb{R}^N$. Consider now a multi-index $\alpha$, $|\alpha|=2m$, and so
$$
\partial^\alpha u = R_{j_1} R_{k_1} \cdots R_{j_{m-1}} R_{k_{m-1}} T_{j,k}(f),
$$
where $R_{j_n}$ is the Riesz transform with respect to the ${j_n}-$th coordinate, $1 \le j_n, n \le N$.
This operator is a product of singular integral operators and therefore a singular integral operator itself.
This completes the proof in the case $N > 2m$.

In the case $N=2m$ it is enough to consider $G(x)=C_N \log|x|$ and
$$
\Delta G(x) = C_N \frac{N-2}{|x|^2},
$$
and to apply the same reasoning as before.
\end{proof}

\begin{corollary}\label{corexun}
The unique solution found in Proposition~\ref{exun} fulfils:
\begin{itemize}
\item $u \in \dot{W}^{2m-\epsilon,N p/(N-\epsilon p)}(\mathbb{R}^N) \, \forall \, 0 \le \epsilon \le 2m$ in case (a).
\item $u \in \dot{W}^{2m-\epsilon,N/(N-\epsilon)}(\mathbb{R}^N) \, \forall \, 0 \le \epsilon \le 2m$ in case (c).
\item $u \in \dot{W}^{N-\epsilon,N/(N-\epsilon)}(\mathbb{R}^N) \, \forall \, 0 \le \epsilon \le N$ in case (d).
\item $D^{2m}u \in \mathcal{H}^1(\mathbb{R}^N)$ in cases (c) and (d).
\end{itemize}
\end{corollary}

\begin{remark}
The strict inequality $p < N/(2m)$ in case (a) of Proposition~\ref{exun} is sharp, see~\cite{grafakos2}.
\end{remark}

\begin{remark}
Note that for an odd $N < 2m$ the formula for $G$ is still given by the second line of~\eqref{funsol}.
For an even $N < 2m$ we have
$$
G(x)=\frac{(-1)^{m-N/2-1}}{N V_N 4^{m-1} \Gamma(N/2)(m-N/2)!(m-1)!} \frac{\log|x|}{|x|^{N-2m}}.
$$
In particular, note that $G$ never decays to zero when $|x| \to \infty$ whenever $N \le 2m$.
\end{remark}

\begin{remark}\label{ung}
Following the previous remark, note that $G$ is not unique since its property~\eqref{deltasol} is invariant with respect
to the addition of a $m-$polyharmonic function. However, if we consider the condition $G \to 0$ when $|x| \to \infty$,
then the above formulas become the unique solution whenever $N > 2m$, and the set of solutions becomes empty if $N \le 2m$.
Moreover, it is not clear how to fix uniqueness in this latter case~\cite{GGS}. In consequence, it is clear that formula~\eqref{convolution}
gives the unique solution to problem~\eqref{linear} for $N > 2m$. For $N=2m$ we take this formula as the definition of unique solution, but see
Remark~\ref{remun} below.
\end{remark}

\begin{lemma}\label{liouville}
Let $v$ be a $m-$harmonic function in $\mathbb{R}^N$. If $v \in \text{BMO}(\mathbb{R}^N)$, then $v$ is constant.
\end{lemma}

\begin{proof}
By definition, $v$ being $m-$harmonic means $\left( -\Delta \right)^m v=0$. Transforming Fourier this equation yields
$$
|k|^{2m} \hat{v}(k)=0,
$$
and since $v \in \text{BMO}(\mathbb{R}^N)$ then $\hat{v}(k) \in \mathcal{S}^*(\mathbb{R}^N)$, where $\mathcal{S}^*(\mathbb{R}^N)$ denotes
the space of Schwartz distributions. This equation implies the support of $\hat{v}$
$$
\text{supp}(\hat{v}) \subset \{0\},
$$
and therefore
$$
\hat{v} = \sum_{|\alpha| \le \ell} C_\alpha \partial^\alpha \delta_0,
$$
for some $\ell \in \mathbb{N}$, $C_\alpha \in \mathbb{R}$, and where $\alpha$ denotes a $N-$dimensional multi-index.
Consequently $v$ is polynomial of degree $\ell$ or lower.
We conclude invoking the John-Nirenberg theorem,
that implies that functions showing a super-logarithmic growth do not belong to $\text{BMO}(\mathbb{R}^N)$,
see~\cite{grafakos2}.
\end{proof}

\begin{remark}
The proof of Lemma~\ref{liouville} actually implies that any $m-$harmonic function in $\mathbb{R}^N$ showing a
sub-linear growth when $|x| \to \infty$ is constant.
\end{remark}

\begin{remark}\label{remun}
Following Remark~\ref{ung}, we note that a way to fix the uniqueness of the fundamental solution in the critical case $N=2m$ is to impose an at most
logarithmic growth when $|x| \to \infty$. According to Lemma~\ref{liouville} this fixes the fundamental solution except for the presence of
an additive constant. Of course, as we are looking for solutions in $\text{BMO}(\mathbb{R}^N)$, and the seminorm of this space is invariant with respect to
the addition of a constant, this fixes uniqueness in the corresponding quotient space in which this seminorm becomes a norm.
In other words, the solution to~\eqref{linear}, $u= \lambda \, G * f$, is unique even if we considered $G$ as a one-parameter family of fundamental solutions
indexed by an additive constant, given that functions in the Hardy space $\mathcal{H}^1(\mathbb{R}^N)$ have zero mean.
Note also that our definition of solution does not guarantee \textit{a priori} that the solution will obey the ``boundary condition'' in any reasonable
sense. However, it obeys it in the pointwise sense, which is the strongest possible sense.
This is justified by Theorem~\ref{higherreg} and Corollary~\ref{corc0} below.
\end{remark}

\section{A topological fixed point theorem}
\label{topofpt}

We now state the fixed point theorem that will allow us to construct the existence theory for our partial differential equation.
This result can be regarded as a corollary of the more general Schauder-Tychonoff theorem~\cite{bonsall}. For the reader
convenience we include a proof of the result, which is independent of the proof present in~\cite{bonsall}.

\begin{theorem}\label{fpt}
Let $\mathcal{Y}$ be a real dual Banach space with separable predual and let $\Upsilon \subset \mathcal{Y}$ be non-empty, convex and weakly$-*$ sequentially compact.
If there exist a weakly$-*$ sequentially continuous map $Z: \Upsilon \longrightarrow \Upsilon$ then $Z$ has at least one fixed point.
\end{theorem}

\begin{proof}
By our hypothesis, every convex, bounded and weakly$-*$ sequentially closed set in $\mathcal{Y}$ is compact (by the Theorem of Banach-Alaoglu)\footnote{Note, however, that strongly closed, convex and bounded is not enough. To see this, consider $\mathcal{Y}=\textbf{M}(\mathbb{R}^n)$ the space of finite Radon measures, which is the dual of $(C_0(\mathbb{R}^n),\|\cdot\|_\infty)$. Now, consider the map $T:\mathcal{Y}\mapsto\mathcal{Y}$ given by $T(\mu)=\mu\ast\mu$. It is not difficult to show that this non-linear map is weak$-*$ sequentially continuous, and maps the simplex $S=\{\mu|\,\mu\ge0,\|\mu\|_\mathcal{Y}=1\}$ into itself. This is a convex weak$-*$ closed and bounded set, and $T$ maps $S$ into itself; the delta function is the unique fixed point of it, but $T$ also maps into itself $S'=\{\mu|\,\mu\ge0,\|\mu\|_\mathcal{Y}=1,\,\mu\ {\rm absolutely\ continuous\ w.r.t.}\,dx\}=L^1(\mathbb{R},dx)\cap S$, which is strongly closed, convex and bounded, but without fixed points.},
and moreover, the trace over that set of the weak$-*$ topology is metrizable. As a result, such a set can be considered a compact metrizable space with respect to that topology; notice in particular that compactness is equivalent to sequential compactness for such $\Upsilon$.

Let us recall how this metric is defined: if we denote by $\mathcal{X}^\ast \equiv \mathcal{Y}$ our dual Banach space,
and $\{y_n\}_{n\ge1}$ is a denumerable dense subset of the closed unitary ball $B$ of the predual $\mathcal{X}$,
we define another seminorm $\|\cdot\|^\ast$ in $\mathcal{X}^\ast$ as
$$
\|x\|^\ast=\sum_{n\ge1}2^{-n}|\langle x,y_n\rangle|,\ x\in \mathcal{X}^\ast.
$$
It is readily checked that the standard norm $\|\cdot\|_{\mathcal{X}^\ast}$ dominates this seminorm, and because of the density of the set $\{y_n\}_{n\ge1}$ over the unit ball of $\mathcal{X}$ and the fact that the weak$-*$ topology is Hausdorff, it is indeed a norm, and it is not hard to prove that it induces the weak$-*$ topology over strongly closed balls of $\mathcal{X}^\ast$, or, more generally, over strongly closed convex sets of $\mathcal{X}^\ast$ (which are known to be weak$-*$ sequentially compact).
Now, since $\Upsilon$ is weakly$-*$ compact then it is totally bounded in the metric which induces the weak$-*$ topology and also bounded with respect to the strong or norm topology. Therefore for any $\delta>0$ we may choose a finite set
$\{v_1, \cdots, v_{n_\delta} | v_i \in \Upsilon, 1 \le i \le n_\delta \}$ such that
$$
\Upsilon \subset \bigcup_{1\le i\le n_\delta} B_{v_i}(\delta),
$$
where $B_{v_i}(\delta)$ is the open ball in $\mathcal{Y}$ (open with respect to the metric induced by $\|\cdot\|^\ast$) whose center is $v_i$ and whose radius is $\delta$.
Consider
$$
\Upsilon_\delta := \left\{ \left. \sum_{i=1}^{n_\delta} c_i v_i \right| c_i \in \mathbb{R}_+ \wedge \sum_{i=1}^{n_\delta} c_i=1 \right\}.
$$
The convexity of $\Upsilon$ guarantees $\Upsilon_\delta \subset \Upsilon$. We introduce the projector
$\mathcal{P}_\delta: \Upsilon \longrightarrow \Upsilon_\delta$,
$$
\mathcal{P}_\delta[v]:= \frac{\sum_{i=1}^{n_\delta} \lambda_i(v)v_i}
{\sum_{i=1}^{n_\delta} \lambda_i(v)},\ \lambda_i(v):=d\left(v,\Upsilon\setminus B_{v_i}(\delta) \right),
$$
where $d(\cdot,\cdot)$ is the distance induced by the norm $\|\cdot\|^\ast$. Any of the functions $\lambda_i(v)$ is Lipschitz continuous and non-negative, and at least one of these functions is positive:
indeed, if $v\in B_{v_i}(\delta)$, then, it is immediate that $\lambda_i(v)\ge\delta$.

Therefore the sum of all of them is positive, and we obtain as a result that
this projection is well defined and continuous for $v \in \Upsilon$. Moreover, as a consequence of the triangle inequality, we have, for $v\in\Upsilon$,
\begin{equation}\label{projineq}
\| \mathcal{P}_\delta[v] -v \|^\ast \le \frac{\sum_{i=1}^{n_\delta} \lambda_i(v) \|v_i-v\|^\ast}
{\sum_{i=1}^{n_\delta} \lambda_i(v)} \le \delta,
\end{equation}
since, for a given $1\le i\le n_\delta$, either $v\in B_{v_i}(\delta)$, in whose case $\|v-v_i\|^\ast<\delta$ or else $v\notin B_{v_i}(\delta)$, in whose case $\lambda_i(v)=0$
(meaning that $\mathcal{P}_\delta[v]$ can be thought of as an small perturbation of the identity map over the set $\Upsilon$ in the metric induced by $\|\cdot\|^\ast$); it is clear also that $\mathcal{P}_\delta[v]$ maps the set $\Upsilon$ to the finite-dimensional set $\Upsilon_\delta$.

Now we define the map $Z_\delta: \Upsilon_\delta \longrightarrow \Upsilon_\delta$,
$$
Z_\delta(v):=\mathcal{P}_\delta[Z(v)],
$$
which is well defined whenever $v \in \Upsilon_\delta$ and continuous. Since $\Upsilon_\delta$ is the closed convex hull of the set $\{v_1,\cdots,v_{n_\delta}\}$
then it is homeomorphic to the closed unit ball in $\mathbb{R}^{j_\delta}$ for some $j_\delta \le n_\delta$.
Now invoke the Brouwer fixed point theorem~\cite{munkres} to see there exists at least one fixed point,
$v_\delta \in \Upsilon_\delta$, of $Z_\delta$.

Taking a sequence $0<\delta_k \to 0$ and select for each $k\ge1$ a fixed point $v_k\in\Upsilon_{\delta_k}\subset \Upsilon$ of $Z_{\delta_j}$. By weak$-*$ compactness of $\Upsilon$, there exists a subsequence $v_{k_j},\,j\ge1$ of the sequence $v_k,\,k\ge1$ which is weak$-*$ convergent to some $v\in\Upsilon$, or in other terms, $\|v-v_{k_j}\|^\ast\rightarrow 0,\,j\rightarrow\infty$. Let us check that $v$ is a fixed point of $Z$:
\begin{eqnarray*}
\|v-Z(v)\|^\ast&=&\|(v-v_{k_j})+(P_{\delta_{k_j}}(Z(v_{k_j}))-Z(v_{k_j})))+(Z(v_{k_j})-Z(v))\|^\ast\\ &&[{\rm since}\ v_{k_j}=Z_{\delta_{k_j}}(v_{k_j})=P_{\delta_{k_j}}(Z(v_{k_j}))]\\
&\le&\|v-v_{k_j}\|^\ast+\|P_{\delta_{k_j}}(Z(v_{k_j}))-Z(v_{k_j})\|^\ast+\|Z(v_{k_j})-
Z(v)\|^\ast\\
&\le& \|v-v_{k_j}\|^\ast+\delta_{k_j}+\|Z(v_{k_j})-
Z(v)\|^\ast\\
&&[{\rm by\ equation~\eqref{projineq}}]\\
&\rightarrow0&,\,j\rightarrow\infty,
\end{eqnarray*}
where, in the last step, we use the weak$-*$ sequential continuity of the map $Z$. So, $\|v-Z(v)\|^\ast=0$, which is equivalent to $v=Z(v)$, as claimed.
\end{proof}

\section{Refinement of the critical Sobolev embedding}
\label{refsobolev}

In this section we introduce a series of preparatory results which are needed in our existence proofs. These
constitute in fact a refinement of the classical Sobolev embedding at the critical dimensional index. Consequently,
this section has an interest on its own, and therefore we have written it in a self-contained fashion.

\begin{theorem}\label{refsobem}
Consider the homogeneous Sobolev space $X=\dot{W}^{1,N}(\mathbb{R}^N)=\{f\in S'(\mathbb{R}^N): |\nabla f|\in L^N(\mathbb{R}^N)\}$, normed by $\|f\|_X=\||\nabla f|\|_{L^N(\mathbb{R}^N)}$. Then we have for all spatial dimensions $N \ge 1$:
\begin{enumerate}
	\item There exists a finite constant $C$ such that for all $f\in X$,
	$$
	\|f\|_{\text{BMO}(\mathbb{R}^N)}\le C \|f\|_X.
	$$
	\item If, in addition, $|\nabla f|\in \mathcal{H}^N(\mathbb{R}^N)$, we have $f\in \text{VMO}(\mathbb{R}^N)$. In any event there exists some absolute and finite $C$, such that given a ball $B=B_r(x_0),\,r>0,x_0\in\mathbb{R}^N$,
	$$
	|f-f_B|_B\le C\||\nabla f|\|_{L^N(B)};\ f_B:=\frac{1}{|B|}\int_B f \,dx.
	$$
\end{enumerate}
\end{theorem}

\begin{remark}
While Part (1) of this theorem is classical, we shall give a proof of it for the sake of completeness.
\end{remark}

\begin{remark}
As $|\nabla f|^Ndx$ can be regarded as a finite and absolutely continuous measure with respect to Lebesgue measure $dx$,
for any $\varepsilon>0$, $\exists \, \delta>0$ such that if $0<r\le\delta$, $|f-f_B|_B\le\varepsilon$, where $r$ is the radius of $B$.
\end{remark}

\begin{remark}\label{remvmo}
For any dimension $N \ge 2$, $\mathcal{H}^N(\mathbb{R}^N)=L^N(\mathbb{R}^N)$. So, an immediate corollary of this theorem can be stated as follows: $\forall \, N \ge 2,\,\dot{W}^{1,N}(\mathbb{R}^N)\subseteq \text{VMO}(\mathbb{R}^N)$, with continuous inclusion.
\end{remark}

\begin{remark}
Note on the other hand that $\mathcal{H}^1(\mathbb{R}^N) \subsetneq L^1(\mathbb{R}^N)$. It is also easy to find functions
$f:\mathbb{R} \longrightarrow \mathbb{R}$ such that $f \in \dot{W}^{1,1}(\mathbb{R})$ and
$f \not\in \text{VMO}(\mathbb{R})$ (such as $f(\cdot)=\arctan(\cdot)$). But however it holds that
$\dot{W}^{1,1}(\mathbb{R}) \subset AC(\mathbb{R}) \cap L^\infty(\mathbb{R})$.
\end{remark}

\begin{remark}
The space $\text{VMO}(\mathbb{R}^N)$ can be defined either intrinsically as the space of those $\text{BMO}(\mathbb{R}^N)$ functions such that for any given $\varepsilon>0$, there exists $\delta>0$ and $R>0$ such that if a ball $B=B_r(x_0)$ has radius smaller that $\delta$ or bigger than $R$, then $|f-f_B|_B\le\varepsilon$ or extrinsically as the closure of the space $C_0(\mathbb{R}^N)$ under the $\text{BMO}(\mathbb{R}^N)$ norm; as Claim (2) of our theorem shows, any function in our space $X$ is very close to be a $\text{VMO}(\mathbb{R}^N)$ function and the averages of the mean oscillation over small balls are always small. This is an intrinsical estimate, but to close the proof of the claim we shall hinge on the extrinsical description of $\text{VMO}(\mathbb{R}^N)$ instead.
\end{remark}

\begin{proof} The key ingredient in Part (1) of the above Theorem is Poincar\'e inequality: given a ball $B$ and an exponent $1\le p\le\infty$, we have, for some finite $C=C(p,B)$,
\begin{equation}\label{poincare}
\|f-f_B\|_{L^p(B)}\le C(p,B)\||\nabla f|\|_{L^p(B)},\ f\in C^1(B).
\end{equation}
The above inequality can be closed to all the (inhomogeneous) Sobolev spaces $W^{1,p}(B)$ in the range $1\le p<\infty$ by an standard density argument; in the case $p=N$, it is easily checked that equation~\eqref{poincare} is scale invariant, meaning that the constant $C_N(B):=C(N,B)$ indeed only depends on $N$, and not on the ball $B_r(x_0)$ we are in. In other words, we have
\begin{equation}\label{critpoin}
\|f-f_B\|_{L^N(B)}\le C_N \||\nabla f|\|_{L^N(B)},\ f\in X.
\end{equation}
From this, the continuous embedding in Claim (1) follows: fix $f\in X$ and $B$ a ball in $\mathbb{R}^N$. Then we have
\begin{eqnarray*}
|f-f_B|_B&\le& \|f-f_B\|_{L^N(B)} \\
&\le& C_N\||\nabla f|\|_{L^N(B)} \\
&\le& C_N\||\nabla f|\|_{L^N(\mathbb{R}^N)},
\end{eqnarray*}
where the first inequality follows by H\"{o}lder inequality and the second by~\eqref{critpoin};
so taking the supremum over all balls in $\mathbb{R}^N$ we find Claim (1) of our theorem follows and moreover the same argument yields the sharper estimate $|f-f_B|_B\le C_N\||\nabla f|\|_{L^N(B)}$.

Now we remind the definition of the (real) Hardy space
$\mathcal{H}^p(\mathbb{R}^N),\,0<p<\infty$; first fix a bump function $\varphi\in C^\infty_c(\mathbb{R}^N)$ with total mass one, and consider the mollifiers $\varphi_t:=t^{-n}\varphi(t^{-1} \, \cdot),\,t>0$. Then we have the following:

\begin{definition}
The Hardy space $\mathcal{H}^p(\mathbb{R}^N)$ is the space of those tempered distributions $f \in S^*(\mathbb{R}^N)$ such that the maximal operator
$$
M^\ast f=\sup_{t>0}|(\varphi_t\ast f)|\in L^p(\mathbb{R}^N).
$$
\end{definition}

\begin{remark}
Notice that this definition in fact does not depend on the choice of $\varphi$.
\end{remark}

Now we use the following Lemmata:

\begin{lemma}\label{density}
For $0<p<\infty$, the space ${\mathcal D}$ of Schwartz functions such that $\hat{f}$ is supported away from the origin is dense in $\mathcal{H}^p(\mathbb{R}^N)$.
\end{lemma}

\begin{proof} We begin with the case $1<p<\infty$. Then $\mathcal{H}^p(\mathbb{R}^N)=L^p(\mathbb{R}^N)$, as a Corollary of the $L^p$ boundedness of the Hardy-Littlewood Maximal operator (which dominates pointwise the auxiliary $M^\ast f$ maximal operator).
If we define $S_t(f):=f\ast\varphi_t$, as it is the convolution of a Schwartz distribution and a Schwartz function, it is $C^\infty$ (see, e.~g., Grafakos~\cite{grafakos2}); and $S_t(f)\in L^p(\mathbb{R}^N)\cap C^\infty(\mathbb{R}^N)$ because $|S_t(f)(x)|\le M^\ast f(x)$. Since $S_t(f)\rightarrow f,\,t\searrow 0$, both in $L^p(\mathbb{R}^N)$ and pointwise almost everywhere (which is a corollary of the Lebesgue Differentiation Theorem and the Dominated Convergence Theorem), it follows that $L^p(\mathbb{R}^N)\cap C^\infty(\mathbb{R}^N)$ is dense in $L^p(\mathbb{R}^N)$. Fix now $\Theta\in C^\infty_c(\mathbb{R}^N)$ such that $\Theta=1$ if $|x|\le1/2$ and $\Theta=0$ if $|x|\ge 1$ and consider the operator
	$$
	R_s(f)(x)=f(x)\Theta(sx),\ s>0.
	$$
It is immediate that $R_s(f)\rightarrow f,\,s\searrow 0$, again both in $L^p$ and pointwise. Moreover, $R_s(f)\rightarrow 0,\,s\rightarrow \infty$, in $L^p$ and pointwise for $x \neq 0$. For a given $\varepsilon>0$, $\exists \, t>0$ such that $\|f-S_t(f)\|_{L^p(\mathbb{R}^N)}\le\varepsilon/2$. For such $t>0$, $\exists \, s>0$ such that $\|S_t(f)-R_s[S_t(f)]\|_{L^p(\mathbb{R}^N)}\le\varepsilon/2$ so that $\|f-R_s[S_t(f)]\|_{L^p(\mathbb{R}^N)}\le\varepsilon$. As $R_s[S_t(f)]\in C^\infty_c(\mathbb{R}^N)\subset S(\mathbb{R}^N)$, it follows that $S(\mathbb{R}^N)$ is dense in $L^p(\mathbb{R}^N),\,1<p<\infty$.

Fix $\varepsilon>0$. Then, $\exists \, g\in S(\mathbb{R}^N)$ with $\|f-g\|_p\le\varepsilon/2$. Consider the operators $M_s(f)$
given by $[M_s(f)]^\wedge:= \hat{f}-R_s(\hat{f})=[1-\Theta(s \, \cdot)]\hat{f}$,
so ${\rm supp}\,[M_s(f)]^\wedge \subset \{\xi\in\mathbb{R}^N:|\xi|\ge 1/(2s)\}$. By Fourier Inversion
$$
M_s(f)=f-\left(\check{\Theta}_{s} \ast f \right); \quad \check{\Theta}_{s}(\cdot):= s^{-N} \check{\Theta}(s^{-1}\, \cdot).
$$
Since the Fourier transform preserves $S(\mathbb{R}^N)$,
it follows that $M_s(f)\in{\mathcal D},\,s>0$, if $f \in S(\mathbb{R}^N)$.
And since $\check{\Theta}_{t},t>0$, define, like the family $\varphi_t$, a standard approximation of identity, it follows that $M_s(f)\rightarrow f$ in $L^p$ as $s \rightarrow \infty$. Picking $s>0$ so that $\|h-M_s(h)\|_p\le\varepsilon/2$, we obtain $\|f-M_s(h)\|_p\le\varepsilon$, which concludes the proof of the Lemma in the range $1<p<\infty$.

In the case $0<p\le 1$, the result follows as a corollary of the Littlewood-Paley square function characterization of the spaces $\mathcal{H}^p(\mathbb{R}^N)$; we refer to Grafakos~\cite{grafakos2}, Chapter 6, for the details.
\end{proof}

\begin{lemma}\label{boundedness}
Let $\Lambda=(-\Delta)^{1/2}$ in the spectral sense (see also section~\ref{nproblems}). Then, for any $N \ge 1$,
	$$
	\Lambda^{-1}:\, \mathcal{H}^N(\mathbb{R}^N) \longrightarrow \text{BMO}(\mathbb{R}^N)
	$$
boundedly.
\end{lemma}

\begin{proof} Given $f\in \mathcal{H}^N(\mathbb{R}^N)$ and a ball $B$ in $\mathbb{R}^N$, for $N\ge1$, using the H\"{o}lder inequality and the Poincar\'{e} inequality for the exponent $N$,
	\begin{eqnarray*}
	|\Lambda^{-1}f-(\Lambda^{-1}f)_B|_B&\le& C_N \left\| \left| \nabla[\Lambda^{-1}(f)] \right| \right\|_{L^N(B)}\\
	&\le& C_N \left\| \left(\sum_{j=1}^{N}|R_j(f)|^2 \right)^{1/2}\right\|_{L^N(\mathbb{R}^N)} \\
	&\le& C'_N\|f\|_{\mathcal{H}^N(\mathbb{R}^N)},
	\end{eqnarray*}
where $R_j$ is the $j-$th Riesz Transform. The last estimate follows since it is a classical result in Fourier Analysis that $\left\| \left(\sum_{j=1}^{N}|R_j(f)|^2 \right)^{1/2}\right\|_{L^p(\mathbb{R}^N)}$ is equivalent to the norm of the Hardy space $\mathcal{H}^p(\mathbb{R}^N)$ in any dimension $N$ and for any exponent $1\le p<\infty$ (we refer again to Grafakos~\cite{grafakos2}).
\end{proof}

Now we can finish the proof of the main theorem of this section: given $f\in X$, there exists a sequence $g_j\in{\mathcal D}$ such that $g_j\rightarrow \Lambda f,\,j\rightarrow\infty$ in $\mathcal{H}^N(\mathbb{R}^N)$ (by Lemma~\ref{density}). Now because $g_j\in{\mathcal D}$, $[\Lambda^{-1}(g_j)]^{\wedge}(\xi)=c_N|\xi|^{-1} \hat{g_j}(\xi);\,\xi\ne 0$. Since for $g\in{\mathcal D}\subset L^p(\mathbb{R}^N),\,1\le p\le\infty$, $\Lambda^{-1}f\in L^q(\mathbb{R}^N),\,q>N$ by the classical Sobolev Embedding Theorem, and this rules out the possibility of a singular support at $\xi=0$ of $(\Lambda^{-1}g)^{\wedge}$. As a result, for $g\in{\mathcal D}$, $(\Lambda^{-1}g)^\wedge$ is also in ${\mathcal D}$; it follows that $\Lambda^{-1}g$ is a Schwartz function, and since the Fourier transform preserves this class, so it belongs too to $S(\mathbb{R}^N)\subset C_0(\mathbb{R}^N)$. Since by Lemma~\ref{boundedness}, $\Lambda^{-1}:\, \mathcal{H}^N(\mathbb{R}^N) \rightarrow
\text{BMO}(\mathbb{R}^N)$ is continuous, given $f\in X$, $f$ belongs to the closure of $C_0(\mathbb{R}^N)$ in $\text{BMO}(\mathbb{R}^N)
$, which is $\text{VMO}(\mathbb{R}^N)$.
\end{proof}

\section{Existence results}
\label{existence}

Now we introduce the general theoretical framework in which our existence results follow. For the sake of clarity, we divide this section
into three subsections corresponding each to the different type of data we are interested in. Our key theoretical tool will be the combination
of the results we have proven in the previous sections with suitable weak continuity properties of the $k-$Hessian. We note that related properties
were studied in the past by several authors, see for
instance~\cite{ball,brezis1,brezis2,brezis3,brezis4,brezis5,coifman,dacorogna,giaquinta,hajlasz,iwaniec1,iwaniec2,iwaniec3,morrey,
muller1,muller2,muller3,muller4,muller5,reshetnyak}.

\subsection{$\mathcal{H}^1$ data}

We start this first subsection introducing a series of technical results which will be of use in the remainder of the section.

\begin{lemma}\label{mvalues}
If $\psi \in \dot{W}^{2m-\delta,N/(N-\delta)}(\mathbb{R}^N) \, \forall \, 0 \le \delta \le 2m-2$ for $m=1 + N (k-1)/(2k) \in \mathbb{N}$ then $S_k[\psi] \in \mathcal{H}^1(\mathbb{R}^N)$.
\end{lemma}

\begin{proof}
It is clear that $S_k[\psi] \in L^1(\mathbb{R}^N)$ for $\psi \in \dot{W}^{2m-\delta,N/(N-\delta)}(\mathbb{R}^N) \, \forall \, 0 \le \delta \le 2m-2$ as a direct consequence of a suitable Sobolev embedding when necessary.
The improved regularity in the statement follows from the divergence form of the $k-$Hessian (see equation~\eqref{divform} below)
and Theorem~I in~\cite{grafakos1}, see also~\cite{grafakos,coifman}.
\end{proof}

\begin{remark}\label{remvalues}
We find admissible values of $m$ whenever
\begin{itemize}
\item $N$ is a multiple of $2k$.
\item $N$ is odd, $N$ is a multiple of $k$ and $k$ is odd.
\end{itemize}
For example, when $N=2k$ we always find the admissible value $m=N/2$. Note also that, as we are assuming $N, k \ge 2$, then $m \ge 2$,
so we are always treating with polyharmonic, rather than harmonic, problems.
\end{remark}

\begin{proposition}\label{wcont}
$S_k[\cdot]$ is weakly$-*$ sequentially continuous from $\dot{W}^{2m,1}(\mathbb{R}^N)$ to the Hardy space $\mathcal{H}^1(\mathbb{R}^N)$,
provided $m=1 + N (k-1)/(2k) \in \mathbb{N}$. That is, if
$$
\psi_n \rightharpoonup \psi; \qquad \text{weakly in} \, \dot{W}^{2m,1}(\mathbb{R}^N),
$$
then
$$
S_k[\psi_n] \overset{*}{\rightharpoonup} S_k[\psi]; \qquad \text{weakly$-*$ in} \, \mathcal{H}^1(\mathbb{R}^N).
$$
\end{proposition}

\begin{proof}
Since $\left[\text{VMO}(\mathbb{R}^N)\right]^*=\mathcal{H}^1(\mathbb{R}^N)$ the statement means that whenever $\psi_n \rightharpoonup \psi$
weakly in $\dot{W}^{2m,1}(\mathbb{R}^N)$
then
$$
\int_{\mathbb{R}^N} \varphi \, S_k[\psi_n] \, dx \to \int_{\mathbb{R}^N} \varphi \, S_k[\psi] \, dx \quad \forall \, \varphi \in \text{VMO}(\mathbb{R}^N).
$$
Note that $S_k[\psi_n], S_k[\psi] \in \mathcal{H}^1(\mathbb{R}^N)$ by Lemma~\ref{mvalues}.
We start proving weak sequential continuity in the sense of distributions
\begin{equation}\label{wcdist}
\psi_n \rightharpoonup \psi \quad \text{weakly in} \, \dot{W}^{2m,1}(\mathbb{R}^N) \Rightarrow
S_k[\psi_n] \to S_k[\psi] \quad \text{in} \, \mathcal{D}^*(\mathbb{R}^N).
\end{equation}
Fix $\phi \in C^\infty_c(\mathbb{R}^N)$ and compute
\begin{equation}\label{intparts}
\int_{\mathbb{R}^N} \phi \, S_k[\psi_n] \, dx = - \frac{1}{k} \sum_{i,j} \int_{\mathbb{R}^N} \phi_i (\psi_n)_j S^{ij}_k[\psi_n] \, dx,
\end{equation}
where we have used integration by parts and the divergence form of the $k-$Hessian
\begin{equation}\label{divform}
S_k[\psi]=\frac{1}{k}\sum_{i,j} \partial_{x_i}(\psi_{x_j} S_k^{ij}[\psi]),
\end{equation}
see~\cite{wang}, where
$$
S^{ij}_k(D^2 \psi) = \left. \frac{\partial}{\partial a_{ij}} \sigma_k[\Lambda(A)] \right|_{A=D^2 \psi},
$$
where $\Lambda(A)$ are the eigenvalues of the $N \times N$ matrix $A$ which entries are $a_{ij}$, and we remind
the definition of the $k-$Hessian $S_k[\psi]=\sigma_k(\Lambda)$ where
$$
\sigma_k(\Lambda)= \sum_{i_1<\cdots<i_k} \Lambda_{i_1} \cdots \Lambda_{i_k},
$$
is the $k^{\mathrm{th}}$ elementary symmetric polynomial and $\Lambda=(\Lambda_1,\cdots,\Lambda_n)$ are the eigenvalues of the Hessian matrix $(D^2 \psi)$.
Now
\begin{eqnarray}\nonumber
\lim_{n \to \infty} \int_{\mathbb{R}^N} \phi \, S_k[\psi_n] \, dx &=& - \lim_{n \to \infty} \frac{1}{k} \sum_{i,j}
\int_{\mathbb{R}^N} \phi_i (\psi_n)_j S^{ij}_k[\psi_n] \, dx, \\ \nonumber
&=& - \frac{1}{k} \sum_{i,j}
\int_{\mathbb{R}^N} \phi_i (\psi)_j S^{ij}_k[\psi] \, dx, \\ \nonumber
&=& \int_{\mathbb{R}^N} \phi \, S_k[\psi] \, dx,
\end{eqnarray}
where the first and third equalities follow from~\eqref{intparts} and the second from the Rellich-Kondrachov theorem that states that
weak convergence
$$
\psi_n \rightharpoonup \psi \; \text{weakly in} \; \dot{W}^{2m,1}(\mathbb{R}^N)
$$
implies
$$
\psi_n \rightarrow \psi \; \text{strongly in} \; \dot{W}^{2,(2N-k)k(k-1)/[(2N-k)k-2(N-k)]}_{\text{loc}}(\mathbb{R}^N)
$$
and
$$
\psi_n \rightarrow \psi \; \text{strongly in} \; \dot{W}^{1,(2N-k)k/(2N-2k)}_{\text{loc}}(\mathbb{R}^N),
$$
if $k \neq N$. If $k=N$, then
$$
\psi_n \rightarrow \psi \; \text{strongly in} \; \dot{W}^{2,N-1/2}_{\text{loc}}(\mathbb{R}^N)
$$
and
$$
\psi_n \rightarrow \psi \; \text{strongly in} \; \dot{W}^{1,2N-1}_{\text{loc}}(\mathbb{R}^N).
$$
So~\eqref{wcdist} is proven.

Given that $C^\infty_c(\mathbb{R}^N)$ is dense in $\text{VMO}(\mathbb{R}^N)$ we may choose
an approximating family $\varphi_\epsilon \in C^\infty_c(\mathbb{R}^N)$ of $\varphi \in \text{VMO}(\mathbb{R}^N)$ such that
$\| \varphi-\varphi_\epsilon \|_{\text{VMO}(\mathbb{R}^N)} \le \epsilon$ for any $\epsilon >0$.
So it follows that
\begin{eqnarray}\nonumber
\int_{\mathbb{R}^N} \varphi \, S_k[\psi_n] \, dx - \int_{\mathbb{R}^N} \varphi \, S_k[\psi] \, dx &=& \int_{\mathbb{R}^N} \varphi_\epsilon \, S_k[\psi_n] \, dx - \int_{\mathbb{R}^N} \varphi_\epsilon \, S_k[\psi] \, dx \\ \nonumber
& & + \int_{\mathbb{R}^N} (\varphi-\varphi_\epsilon) \, S_k[\psi_n] \, dx \\ \nonumber & & - \int_{\mathbb{R}^N} (\varphi-\varphi_\epsilon) \, S_k[\psi] \, dx.
\end{eqnarray}
Since $S_k[\psi_n]$ and $S_k[\psi]$ are uniformly bounded in $\mathcal{H}^1(\mathbb{R}^N)$, we can estimate
\begin{eqnarray}\nonumber
& & \left| \int_{\mathbb{R}^N} \varphi \, S_k[\psi_n] \, dx - \int_{\mathbb{R}^N} \varphi \, S_k[\psi] \, dx \right|
\\ \nonumber &\le&
\left\{ \| S_k[\psi_n] \|_{\mathcal{H}^1(\mathbb{R}^N)} + \| S_k[\psi] \|_{\mathcal{H}^1(\mathbb{R}^N)} \right\} \\ \nonumber
& & \times \| \varphi-\varphi_\epsilon \|_{\text{VMO}(\mathbb{R}^N)} \\ \nonumber
& & + \left| \int_{\mathbb{R}^N} \varphi_\epsilon \, S_k[\psi_n] \, dx - \int_{\mathbb{R}^N} \varphi_\epsilon \, S_k[\psi] \, dx \right|,
\end{eqnarray}
and
\begin{equation}\nonumber
\limsup_{n \to \infty} \left| \int_{\mathbb{R}^N} \varphi \, S_k[\psi_n] \, dx - \int_{\mathbb{R}^N} \varphi \, S_k[\psi] \, dx \right| \le C \epsilon + o(1).
\end{equation}
The statement follows by the arbitrariness of $\epsilon$.
\end{proof}

\begin{corollary}\label{cwcont}
$S_k[\cdot]$ is weakly$-*$ continuous from the homogeneous Sobolev space
$\dot{W}^{2m-\delta,N/(N-\delta)}(\mathbb{R}^N) \, \forall \, 0 \le \delta < 2m-2$ to $\mathcal{H}^1(\mathbb{R}^N)$,
provided $m=1 + N (k-1)/(2k) \in \mathbb{N}$.
\end{corollary}

\begin{corollary}\label{cwcont2}
The $k-$Hessian $S_k[\cdot]$ is weakly$-*$ continuous from the homogeneous Sobolev space
$\dot{W}^{2m-\delta,N/(N-\delta)}(\mathbb{R}^N) \, \forall \, 0 \le \delta < 2m-2$
to $\bf{M}(\mathbb{R}^N)$, where $\bf{M}(\mathbb{R}^N)$ is the set of (signed) Radon measures,
provided $m=1 + N (k-1)/(2k) \in \mathbb{N}$. That is, if
$$
\psi_n \rightharpoonup \psi; \qquad \text{weakly in} \, \dot{W}^{2m-\delta,N/(N-\delta)}(\mathbb{R}^N),
$$
then
$$
S_k[\psi_n] \overset{*}{\rightharpoonup} S_k[\psi]; \qquad \text{weakly$-*$ in} \, \bf{M}(\mathbb{R}^N).
$$
\end{corollary}

Now we state the main result of this subsection:

\begin{theorem}\label{exh1}
Let $m=1 + N (k-1)/(2k) \in \mathbb{N}$.
Then problem~\eqref{bvphessian}-\eqref{bcons} has at least one weak solution in $\dot{W}^{2m,1}(\mathbb{R}^N)$ for any $N \ge 4$ and any $N/2 \ge k \ge 2$ ($N, k \in \mathbb{N}$)
provided $|\lambda|$ is small enough and $f \in \mathcal{H}^{1}(\mathbb{R}^N)$.
Moreover any such solution $u \in \dot{W}^{2m-\epsilon,N/(N-\epsilon)}(\mathbb{R}^N) \, \forall \, 0 \le \epsilon \le 2m$
and $D^{2m}u \in \mathcal{H}^1(\mathbb{R}^N)$.
\end{theorem}

\begin{proof}
Consider $w \in \dot{W}^{2m,1}(\mathbb{R}^N)$. Then $S_k[-w] \in \mathcal{H}^1(\mathbb{R}^N)$
by Lemma~\ref{mvalues} and the equation
\begin{eqnarray}\nonumber
\left( -\Delta \right)^m u &=& S_k[-w] + \lambda f, \qquad x \in \mathbb{R}^N, \\ \nonumber
u &\to& 0, \quad \text{when} \quad |x| \to \infty,
\end{eqnarray}
has a unique solution $u \in \dot{W}^{2m-\epsilon,N/(N-\epsilon)}(\mathbb{R}^N) \, \forall \, 0 \le \epsilon \le 2m$
such that $D^{2m}u \in \mathcal{H}^1(\mathbb{R}^N)$ by Corollary~\ref{corexun}.
So the map
\begin{eqnarray}\nonumber
\mathcal{T}: \mathcal{H}^1(\mathbb{R}^N) &\longrightarrow&  \mathcal{H}^1(\mathbb{R}^N) \\ \nonumber
v &\longmapsto& v' = S_k \left[\left( -\Delta \right)^{-m} (-v)\right] + \lambda f,
\end{eqnarray}
is well defined and moreover
\begin{eqnarray}\nonumber
\|v'\|_{\mathcal{H}^1(\mathbb{R}^N)} &\ll& \|S_k \left[\left( -\Delta \right)^{-m} (-v)\right]\|_{\mathcal{H}^1(\mathbb{R}^N)}
+ |\lambda| \|f\|_{\mathcal{H}^1(\mathbb{R}^N)} \\ \nonumber
&\ll& \|\left( -\Delta \right)^{-m} (-v)\|_{\dot{W}^{2m,1}(\mathbb{R}^N)}^k + |\lambda| \|f\|_{\mathcal{H}^1(\mathbb{R}^N)} \\ \nonumber
&\ll& \|v\|_{\mathcal{H}^1(\mathbb{R}^N)}^k + |\lambda| \|f\|_{\mathcal{H}^1(\mathbb{R}^N)},
\end{eqnarray}
by the triangle inequality in the first step, Lemma~\ref{mvalues} in the second, and Proposition~\ref{exun} in the third.
Now consider the particular case $v=0$, i.~e. $v_0 = \lambda f$, then obviously
$\|v_0\|_{\mathcal{H}^{1}(\mathbb{R}^N)} = |\lambda| \|f\|_{\mathcal{H}^1(\mathbb{R}^N)}$ and
\begin{equation}\nonumber
v'-v_0 = S_k \left[\left( -\Delta \right)^{-m} (-v)\right], \qquad x \in \mathbb{R}^N.
\end{equation}
Therefore
\begin{eqnarray}\nonumber
\|v'-v_0\|_{\mathcal{H}^1(\mathbb{R}^N)} &=& \|S_k \left[\left( -\Delta \right)^{-m} (-v)\right]\|_{\mathcal{H}^1(\mathbb{R}^N)} \\ \nonumber
&\ll& \|\left( -\Delta \right)^{-m} (-v)\|_{\dot{W}^{2m,1}(\mathbb{R}^N)}^k \\ \nonumber
&\ll& \|v\|_{\mathcal{H}^1(\mathbb{R}^N)}^k \\ \nonumber
&\ll& \left[ \|v - v_0\|_{\mathcal{H}^1(\mathbb{R}^N)}
+ \|v_0\|_{\mathcal{H}^1(\mathbb{R}^N)} \right]^k \\ \nonumber
&=& \left[ \|v - v_0\|_{\mathcal{H}^1(\mathbb{R}^N)}
+ |\lambda| \|f\|_{\mathcal{H}^1(\mathbb{R}^N)} \right]^k.
\end{eqnarray}
Consequently it is clear that $\mathcal{T}$ will map the ball
$$
B = \left\{ v \in \mathcal{H}^1(\mathbb{R}^N) : \|v - v_0\|_{\mathcal{H}^1(\mathbb{R}^N)} \le R \right\}
$$
into itself provided we choose $R$ and $|\lambda|$ small enough.

Now assume $\psi_j \stackrel{*}{\rightharpoonup} \psi$ in $\mathcal{H}^1(\mathbb{R}^N)$, therefore
$$
\left< \left( -\Delta \right)^{m} \phi, \left( -\Delta \right)^{-m} \psi_j \right> \longrightarrow
\left< \left( -\Delta \right)^{m} \phi, \left( -\Delta \right)^{-m} \psi \right>
$$
for any fixed $\phi \in \text{VMO}(\mathbb{R}^N)$, or equivalently
$$
\left< \hat{\phi}, \left( -\Delta \right)^{-m} \psi_j \right> \longrightarrow
\left< \hat{\phi}, \left( -\Delta \right)^{-m} \psi \right>,
$$
for any fixed $\hat{\phi} \in I_{-2m}(\text{VMO})(\mathbb{R}^N)$, with the obvious definition of $I_{-2m}(\text{VMO})(\mathbb{R}^N)$
(see for instance~\cite{stri}). By Corollary~\ref{corexun} $\left( -\Delta \right)^{-m} \psi_j \in \dot{W}^{2m-1,N/(N-1)}(\mathbb{R}^N)$,
but we need $\left( -\Delta \right)^{-m} \psi_j \rightharpoonup \left( -\Delta \right)^{-m} \psi$
in $\dot{W}^{2m-1,N/(N-1)}(\mathbb{R}^N)$; note that the first mode of convergence does not, in principle, trivially imply the second.
On the other hand the two facts $\{\dot{W}^{2m-1,N/(N-1)}(\mathbb{R}^N)\}^*\! = \dot{W}^{1-2m,N}(\mathbb{R}^N)$ and Remark~\ref{remvmo} imply
that, for $N \ge 2$, the first mode of convergence indeed implies the second.
As a consequence of this and Corollary~\ref{cwcont} the map $\mathcal{T}$ is weakly$-*$ continuous, and
consequently by Theorem~\ref{fpt} it has a fixed point.
The existence of solution follows from $u=\left( -\Delta \right)^{-m}v$ and Proposition~\ref{exun}.
The regularity follows by Sobolev embeddings.
\end{proof}

\subsection{Summable Data}
\label{l1data}

An analogous existence theorem can still be proven for data $f \in L^1(\mathbb{R}^N)$.

\begin{theorem}\label{exl1}
Let $m=1 + N (k-1)/(2k) \in \mathbb{N}$.
Then problem~\eqref{bvphessian}-\eqref{bcons} has at least one weak solution in
$\dot{W}^{2m-\epsilon,N/(N-\epsilon)}(\mathbb{R}^N) \, \forall \, 0 < \epsilon \le 2m$
for any $N \ge 8$ and any $N/2 > k \ge 2$ ($N, k \in \mathbb{N}$)
provided $|\lambda|$ is small enough and $f \in L^{1}(\mathbb{R}^N)$.
Moreover any such solution fulfills $D^{2m} u \in L^{1,\infty}(\mathbb{R}^N)$.
\end{theorem}

\begin{proof}
Consider $w \in \dot{W}^{2m-1,N/(N-1)}(\mathbb{R}^N)$. Then $S_k[-w] \in \mathcal{H}^1(\mathbb{R}^N)$
by Lemma~\ref{mvalues} and Remark~\ref{remvalues}, and the equation
\begin{eqnarray}\nonumber
\left( -\Delta \right)^m u &=& S_k[-w] + \lambda f, \qquad x \in \mathbb{R}^N, \\ \nonumber
u &\to& 0, \quad \text{when} \quad |x| \to \infty,
\end{eqnarray}
has a unique solution $u \in \dot{W}^{2m-\epsilon,N/(N-\epsilon)}(\mathbb{R}^N) \, \forall \, 0 < \epsilon \le 2m$
such that $D^{2m}u \in L^{1,\infty}(\mathbb{R}^N)$ by Proposition~\ref{exun}.
So the map
\begin{eqnarray}\nonumber
\mathcal{T}: \dot{W}^{2m-1,N/(N-1)}(\mathbb{R}^N) &\longrightarrow&  \dot{W}^{2m-1,N/(N-1)}(\mathbb{R}^N) \\ \nonumber
w &\longmapsto& u = \left( -\Delta \right)^{-m} S_k \left[-w\right] + \lambda \left( -\Delta \right)^{-m} f,
\end{eqnarray}
is well defined and furthermore for $g:=\left( -\Delta \right)^{-m} f$
\begin{eqnarray}\nonumber
\|u\|_{\dot{W}^{2m-1,N/(N-1)}(\mathbb{R}^N)} &\ll& \left\|\left( -\Delta \right)^{-m} S_k \left[-w\right]\right\|_{\dot{W}^{2m-1,N/(N-1)}(\mathbb{R}^N)} \\
\nonumber & & + |\lambda| \, \|g\|_{\dot{W}^{2m-1,N/(N-1)}(\mathbb{R}^N)} \\ \nonumber
&\ll& \left\| w \right\|_{\dot{W}^{2m-1,N/(N-1)}(\mathbb{R}^N)}^k \\
\nonumber & & + |\lambda| \, \|g\|_{\dot{W}^{2m-1,N/(N-1)}(\mathbb{R}^N)},
\end{eqnarray}
by the triangle inequality and Proposition~\ref{exun} in the first step, and Lemma~\ref{mvalues} and Corollary~\ref{corexun} in the second.
Now consider the particular case $w=0$, i.~e. $u_0 = \lambda \left( -\Delta \right)^{-m} f$, then clearly
$\|u_0\|_{\dot{W}^{2m-1,N/(N-1)}(\mathbb{R}^N)} = |\lambda| \, \|g\|_{\dot{W}^{2m-1,N/(N-1)}(\mathbb{R}^N)}$ and
\begin{equation}\nonumber
u-u_0 = \left( -\Delta \right)^{-m} S_k \left[-w\right], \qquad x \in \mathbb{R}^N.
\end{equation}
Therefore
\begin{eqnarray}\nonumber
\|u-u_0\|_{\dot{W}^{2m-1,N/(N-1)}(\mathbb{R}^N)} &=& \left\|\left( -\Delta \right)^{-m} S_k \left[-w\right]\right\|_{\dot{W}^{2m-1,N/(N-1)}(\mathbb{R}^N)}
\\ \nonumber
&\ll& \left\| w \right\|_{\dot{W}^{2m-1,N/(N-1)}(\mathbb{R}^N)}^k \\ \nonumber
&\ll& \left[ \left\| w - u_0 \right\|_{\dot{W}^{2m-1,N/(N-1)}(\mathbb{R}^N)} \right. \\ \nonumber
& & \left. + \left\| u_0 \right\|_{\dot{W}^{2m-1,N/(N-1)}(\mathbb{R}^N)} \right]^k \\ \nonumber
&=& \left[ \left\| w - u_0 \right\|_{\dot{W}^{2m-1,N/(N-1)}(\mathbb{R}^N)} \right. \\ \nonumber
& & \left. + |\lambda| \, \|g\|_{\dot{W}^{2m-1,N/(N-1)}(\mathbb{R}^N)} \right]^k.
\end{eqnarray}
Consequently it is clear that $\mathcal{T}$ maps the ball
$$
B = \left\{ w \in \dot{W}^{2m-1,N/(N-1)}(\mathbb{R}^N) : \|w - u_0\|_{\dot{W}^{2m-1,N/(N-1)}(\mathbb{R}^N)} \le R \right\}
$$
into itself given that we choose $R$ and $|\lambda|$ small enough.

Corollary~\ref{cwcont} implies the convergence property
\begin{equation}\label{wconl1}
\left< \phi, S_k[\psi_n] \right> \longrightarrow \left< \phi, S_k[\psi] \right>,
\end{equation}
for any fixed $\phi \in \text{VMO}(\mathbb{R}^N)$ given that $\psi_n \rightharpoonup \psi$ in $\dot{W}^{2m-1,N/(N-1)}(\mathbb{R}^N)$.
By equation~\eqref{wconl1} we get
$$
\left< \left( -\Delta \right)^{m} \phi, \left( -\Delta \right)^{-m} S_k[\psi_n] \right> \longrightarrow
\left< \left( -\Delta \right)^{m} \phi, \left( -\Delta \right)^{-m} S_k[\psi] \right>,
$$
for any fixed $\phi \in \text{VMO}(\mathbb{R}^N)$, or in other terms
$$
\left< \hat{\phi}, \left( -\Delta \right)^{-m} S_k[\psi_n] \right> \longrightarrow
\left< \hat{\phi}, \left( -\Delta \right)^{-m} S_k[\psi] \right>,
$$
for any fixed $\hat{\phi} \in I_{-2m}(\text{VMO})(\mathbb{R}^N)$, as in the previous subsection.
This mode of convergence is not, in principle, equivalent to the one we need:
$\left( -\Delta \right)^{-m} S_k[\psi_n] \rightharpoonup \left( -\Delta \right)^{-m} S_k[\psi]$
in $\dot{W}^{2m-1,N/(N-1)}(\mathbb{R}^N)$.
However using $\{\dot{W}^{2m-1,N/(N-1)}(\mathbb{R}^N)\}^*\! = \dot{W}^{1-2m,N}(\mathbb{R}^N)$ and Remark~\ref{remvmo} we find
for $N \ge 2$ that the second mode of convergence follows as a consequence of the first.

Given our assumption $N \ge 2$ we get that the map $\mathcal{T}$ is weakly continuous in $\dot{W}^{2m-1,N/(N-1)}(\mathbb{R}^N)$
(and thus it is weakly$-*$ continuous), so by Theorem~\ref{fpt} it has a fixed point.
The regularity follows from Proposition~\ref{exun}
and a classical bootstrap argument.
\end{proof}

\begin{remark}\label{remeqcla}
Note that the space $\dot{W}^{2m-1,N/(N-1)}(\mathbb{R}^N)$ is not a Banach space since $\|\cdot\|_{\dot{W}^{2m-1,N/(N-1)}(\mathbb{R}^N)}$ is a seminorm rather than a norm.
Note however that the null subspace of $\|\cdot\|_{\dot{W}^{2m-1,N/(N-1)}(\mathbb{R}^N)}$ is composed by the polynomials of degree smaller or equal to $2m-2$.
So we can consider $\dot{W}^{2m-1,N/(N-1)}(\mathbb{R}^N)$ as the quotient space which equivalence classes are closed with respect to the addition of one such
polynomial. Since in Theorem~\ref{exl1} we are proving the existence of solutions that vanish at infinity, and the set of polynomials that vanish at infinity
has a unique element that is identically zero, the use of norm $\|\cdot\|_{\dot{W}^{2m-1,N/(N-1)}(\mathbb{R}^N)}$ in the proof of Theorem~\ref{exl1} is meaningful.
\end{remark}

\subsection{$L^p$ data}

We now state the complementary result that assumes our datum $f \in L^p(\mathbb{R}^N)$.

\begin{theorem}\label{exlp}
Let $m=1 + N (k-1)/(2pk) \in \mathbb{N}$.
Then problem~\eqref{bvphessian}-\eqref{bcons} has at least one weak solution in
$\dot{W}^{2m-\epsilon,N p/(N-\epsilon p)}(\mathbb{R}^N) \, \forall \, 0 \le \epsilon \le 2m$
for any $N \ge 9$ and any $N/2 > k \ge 2$ ($N, k \in \mathbb{N}$)
provided $|\lambda|$ is small enough and $f \in L^{p}(\mathbb{R}^N)$, $1<p<N/(2k)$.
\end{theorem}

\begin{proof}
The proofs mimics that of Theorem~\ref{exh1} with the space $\dot{W}^{2m,p}(\mathbb{R}^N)$
playing the role of both $\dot{W}^{2m-1,N/(N-1)}(\mathbb{R}^N)$ and $\dot{W}^{2m,1}(\mathbb{R}^N)$,
except for the proof of weak convergence. Therefore we will only include this part here.

In this case $\psi \in \dot{W}^{2m,p}(\mathbb{R}^N) \hookrightarrow \dot{W}^{2,k p}(\mathbb{R}^N)$, so we need to prove
$$
\int_{\mathbb{R}^N} \varphi \, S_k[\psi_n] \, dx \to \int_{\mathbb{R}^N} \varphi \, S_k[\psi] \, dx \quad \forall \, \varphi \in L^q(\mathbb{R}^N),
$$
where $p^{-1} + q^{-1} = 1$ (and so $q>1$).
We again start proving weak continuity in the sense of distributions
\begin{equation}\label{wcdist2}
\psi_n \rightharpoonup \psi \quad \text{weakly in} \, \dot{W}^{2m,p}(\mathbb{R}^N) \Rightarrow
S_k[\psi_n] \to S_k[\psi] \quad \text{in} \, \mathcal{D}^*(\mathbb{R}^N).
\end{equation}
We fix $\phi \in C^\infty_c(\mathbb{R}^N)$ and calculate
\begin{equation}\label{intparts2}
\int_{\mathbb{R}^N} \phi \, S_k[\psi_n] \, dx = - \frac{1}{k} \sum_{i,j} \int_{\mathbb{R}^N} \phi_i (\psi_n)_j S^{ij}_k[\psi_n] \, dx,
\end{equation}
where we have used integration by parts and the divergence form of the $k-$Hessian.
Now we take the limit
\begin{eqnarray}\nonumber
\lim_{n \to \infty} \int_{\mathbb{R}^N} \phi \, S_k[\psi_n] \, dx &=& - \lim_{n \to \infty} \frac{1}{k} \sum_{i,j}
\int_{\mathbb{R}^N} \phi_i (\psi_n)_j S^{ij}_k[\psi_n] \, dx, \\ \nonumber
&=& - \frac{1}{k} \sum_{i,j}
\int_{\mathbb{R}^N} \phi_i (\psi)_j S^{ij}_k[\psi] \, dx, \\ \nonumber
&=& \int_{\mathbb{R}^N} \phi \, S_k[\psi] \, dx,
\end{eqnarray}
where the first and third equalities follow from~\eqref{intparts2} and the second from the Rellich-Kondrachov theorem which for
$$
\psi_n \rightharpoonup \psi \; \text{weakly in} \; \dot{W}^{2m,p}(\mathbb{R}^N)
$$
implies
$$
\psi_n \rightarrow \psi \; \text{strongly in} \; \dot{W}^{2,(2N-pk)pk(k-1)/[(2N-pk)pk-2(N-pk)])}_{\text{loc}}(\mathbb{R}^N)
$$
and
$$
\psi_n \rightarrow \psi \; \text{strongly in} \; \dot{W}^{1,(2N-pk)pk/(2N-2pk)}_{\text{loc}}(\mathbb{R}^N).
$$
Thus~\eqref{wcdist2} is proven.

Since $C^\infty_c(\mathbb{R}^N)$ is dense in $L^q(\mathbb{R}^N)$, $p^{-1}+q^{-1}=1$, we select
an approximating family $\varphi_\epsilon \in C^\infty_c(\mathbb{R}^N)$ of $\varphi \in L^q(\mathbb{R}^N)$ such that
$\| \varphi-\varphi_\epsilon \|_{L^q(\mathbb{R}^N)} \le \epsilon$ for any $\epsilon >0$.
So it holds that
\begin{eqnarray}\nonumber
\int_{\mathbb{R}^N} \varphi \, S_k[\psi_n] \, dx - \int_{\mathbb{R}^N} \varphi \, S_k[\psi] \, dx &=& \int_{\mathbb{R}^N} \varphi_\epsilon \, S_k[\psi_n] \, dx - \int_{\mathbb{R}^N} \varphi_\epsilon \, S_k[\psi] \, dx \\ \nonumber
& & + \int_{\mathbb{R}^N} (\varphi-\varphi_\epsilon) \, S_k[\psi_n] \, dx \\ \nonumber & & - \int_{\mathbb{R}^N} (\varphi-\varphi_\epsilon) \, S_k[\psi] \, dx.
\end{eqnarray}
Given that $S_k[\psi_n]$ and $S_k[\psi]$ are bounded in $L^p(\mathbb{R}^N)$, we can establish the estimate
\begin{eqnarray}\nonumber
& & \left| \int_{\mathbb{R}^N} \varphi \, S_k[\psi_n] \, dx - \int_{\mathbb{R}^N} \varphi \, S_k[\psi] \, dx \right|
\\ \nonumber &\le&
\left\{ \| S_k[\psi_n] \|_{L^p(\mathbb{R}^N)} + \| S_k[\psi] \|_{L^p(\mathbb{R}^N)} \right\} \\ \nonumber
& & \times \| \varphi-\varphi_\epsilon \|_{L^q(\mathbb{R}^N)} \\ \nonumber
& & + \left| \int_{\mathbb{R}^N} \varphi_\epsilon \, S_k[\psi_n] \, dx - \int_{\mathbb{R}^N} \varphi_\epsilon \, S_k[\psi] \, dx \right|,
\end{eqnarray}
and
\begin{equation}\nonumber
\limsup_{n \to \infty} \left| \int_{\mathbb{R}^N} \varphi \, S_k[\psi_n] \, dx - \int_{\mathbb{R}^N} \varphi \, S_k[\psi] \, dx \right| \le C \epsilon + o(1).
\end{equation}
Therefore the arbitrariness of $\epsilon$ guarantees that, if
$$
\psi_n \rightharpoonup \psi; \qquad \text{weakly in} \, \dot{W}^{2m,p}(\mathbb{R}^N),
$$
then
$$
S_k[\psi_n] \rightharpoonup S_k[\psi]; \qquad \text{weakly in} \, L^p(\mathbb{R}^N),
$$
and so the statement follows.
\end{proof}

\begin{remark}
The lower bounds for the values of $N$ in Theorems~\ref{exh1},~\ref{exl1} and~\ref{exlp} are easily proven using the inequalities
in the statement of Proposition~\ref{exun}. Also, it is easy to find examples of $m$, $N$, $k$ and $p$ for which the statements of these
theorems apply.
\end{remark}

\section{Local Uniqueness}
\label{locuniq}

In this section we prove existence and local uniqueness of a solution under more restrictive conditions.
We start concentrating on the case that corresponds to Theorem~\ref{exl1}.

\begin{definition}
Let $u$ be a weak solution to problem~\eqref{bvphessian}-\eqref{bcons} and $\mathcal{W}$ a Banach space.
If there exists a $\rho > 0$ such that this solution is unique in the ball
$$
\tilde{B}_\rho(u)=\left\{ v \in \mathcal{W} : \| u - v \|_{\mathcal{W}} \le \rho \right\},
$$
then we say that $u$ is \emph{locally unique in $\mathcal{W}$}.
\end{definition}

\begin{theorem}\label{exul1}
Let $m=1 + N (k-1)/(2k) \in \mathbb{N}$.
Then problem~\eqref{bvphessian}-\eqref{bcons} has at least one weak solution in
$\dot{W}^{2m-\epsilon,N/(N-\epsilon)}(\mathbb{R}^N) \, \forall \, 0 < \epsilon \le 2m$
for any $N \ge 8$ and any $N/2 > k \ge 2$ ($N, k \in \mathbb{N}$)
provided $|\lambda|$ is small enough and $f \in L^{1}(\mathbb{R}^N)$.
Moreover any such solution fulfills $D^{2m} u \in L^{1,\infty}(\mathbb{R}^N)$
and at least one is locally unique in $\dot{W}^{2m-1,N/(N-1)}(\mathbb{R}^N)$.
\end{theorem}

\begin{proof}
Consider $w_1, w_2 \in \dot{W}^{2m-1,N/(N-1)}(\mathbb{R}^N)$. Then $S_k[-w_{1,2}] \in \mathcal{H}^1(\mathbb{R}^N)$ by Lemma~\ref{mvalues}
and the equations
\begin{eqnarray}\nonumber
\left( -\Delta \right)^m u_{1,2} &=& S_k[-w_{1,2}] + \lambda f, \qquad x \in \mathbb{R}^N, \\ \nonumber
u_{1,2} &\to& 0, \quad \text{when} \quad |x| \to \infty,
\end{eqnarray}
have a unique solution $u_{1,2} \in \dot{W}^{2m-1,N/(N-1)}(\mathbb{R}^N)$ by Proposition~\ref{exun}.
Now we can subtract them
\begin{eqnarray}\nonumber
\left( -\Delta \right)^m (u_{1} - u_2) &=& S_k[-w_{1}] - S_k[-w_{2}] , \qquad x \in \mathbb{R}^N, \\ \nonumber
u_{1} - u_2 &\to& 0, \quad \text{when} \quad |x| \to \infty,
\end{eqnarray}
and find a unique solution $u_1 - u_2 \in \dot{W}^{2m-1,N/(N-1)}(\mathbb{R}^N)$ such that
$$
\| u_1 - u_2 \|_{\dot{W}^{2m-1,N/(N-1)}(\mathbb{R}^N)} \ll \| S_k[-w_{1}] - S_k[-w_{2}] \|_1,
$$
by the same proposition.
Now using
$$
S_k[\psi]=\frac{1}{k}\sum_{i,j} \partial_{x_i}(\psi_{x_j} S_k^{ij}[\psi])= \frac{1}{k}\sum_{i,j} \psi_{x_i x_j} S_k^{ij}[\psi],
$$
since
$$
\sum_{i} \partial_{x_i} S_k^{ij}[\psi]=0 \quad \forall \,\, 1 \le j \le N,
$$
for any smooth function $\psi$~\cite{morrey}, yields
\begin{eqnarray}\nonumber
& & \| u_1 - u_2 \|_{\dot{W}^{2m-1,N/(N-1)}(\mathbb{R}^N)} \\ \nonumber
&\ll& \left\| \frac{1}{k}\sum_{i,j} (w_1)_{x_i x_j} S_k^{ij}[w_1]
- \frac{1}{k}\sum_{i,j} (w_2)_{x_i x_j} S_k^{ij}[w_2] \right\|_1 \\ \nonumber
&\ll& \left[ \| w_1 \|_{\dot{W}^{2m-1,N/(N-1)}(\mathbb{R}^N)} + \| w_2 \|_{\dot{W}^{2m-1,N/(N-1)}(\mathbb{R}^N)} \right]^{k-1}
\\ \nonumber & & \times \| w_1 - w_2 \|_{\dot{W}^{2m-1,N/(N-1)}(\mathbb{R}^N)},
\end{eqnarray}
after arguing by approximation in the first step and using Sobolev and triangle inequalities, and a reasoning akin to
that in the proof of Theorem~1 in~\cite{brezis3}, in the second.
We know the map
\begin{eqnarray}\nonumber
\mathcal{T}: \dot{W}^{2m-1,N/(N-1)}(\mathbb{R}^N) &\longrightarrow&  \dot{W}^{2m-1,N/(N-1)}(\mathbb{R}^N) \\ \nonumber
w_{1,2} &\longmapsto& u_{1,2},
\end{eqnarray}
is well defined and also maps the ball
$$
B = \left\{ w \in \dot{W}^{2m-1,N/(N-1)}(\mathbb{R}^N) : \|w - u_0\|_{\dot{W}^{2m-1,N/(N-1)}(\mathbb{R}^N)} \le R \right\}
$$
into itself provided we choose $R$ and $|\lambda|$ small enough by the proof of Theorem~\ref{exh1}.
Therefore
\begin{eqnarray}\nonumber
\| u_1 - u_2 \|_{\dot{W}^{2m-1,N/(N-1)}(\mathbb{R}^N)}
&\ll& \left[ |\lambda| \|f\|_{L^1(\mathbb{R}^N)} + R \right]^{k-1} \\ \nonumber
& & \times \| w_1 - w_2 \|_{\dot{W}^{2m-1,N/(N-1)}(\mathbb{R}^N)}
\\ \nonumber
&<& \frac12 \, \| w_1 - w_2 \|_{\dot{W}^{2m-1,N/(N-1)}(\mathbb{R}^N)},
\end{eqnarray}
where we have used the triangle inequality and Proposition~\ref{exun} in the first step and
have chosen sufficiently smaller $R$ and $|\lambda|$ in the second. Thus the existence and uniqueness of the solution
follows by the application of the Banach fixed point theorem and the regularity by Proposition~\ref{exun} and a classical bootstrap argument.
\end{proof}

We can now state the corresponding result for $f \in L^p(\mathbb{R}^N)$.

\begin{theorem}\label{exulp}
Let $m=1 + N (k-1)/(2pk) \in \mathbb{N}$.
Then problem~\eqref{bvphessian}-\eqref{bcons} has at least one weak solution in
$\dot{W}^{2m-\epsilon,N p/(N-\epsilon p)}(\mathbb{R}^N) \, \forall \, 0 \le \epsilon \le 2m$
for any $N \ge 9$ and any $N/2 > k \ge 2$ ($N, k \in \mathbb{N}$)
provided $|\lambda|$ is small enough and $f \in L^{p}(\mathbb{R}^N)$, $1<p<N/(2k)$.
Moreover, at least one of these solutions is locally unique in $\dot{W}^{2m, p}(\mathbb{R}^N)$.
\end{theorem}

\begin{proof}
Follows analogously to that of Theorem~\ref{exul1}.
\end{proof}

\begin{remark}
The proof of Theorem~\ref{exul1} is not applicable to the case $f \in \mathcal{H}^1(\mathbb{R}^N)$ and $k=N/2$;
for the existence theory under these hypotheses the reader is referred to Theorem~\ref{exh1}. On the other hand assuming
$f \in \mathcal{H}^1(\mathbb{R}^N)$ and $k < N/2$ allows one to reproduce this proof identically with the slight improvement
in regularity $D^{2m} u \in \mathcal{H}^{1}(\mathbb{R}^N)$.
\end{remark}

\section{Nonlocal problems}
\label{nproblems}

In this section we extend our results for problem~\eqref{rkhessian} to
\begin{equation}\label{rkhessiannl}
\Lambda^n u = S_k[-u] + \lambda f, \qquad x \in \mathbb{R}^N,
\end{equation}
where $\Lambda$ is a pseudo-differential operator defined in the following way.

\begin{definition}\label{lambda}
The pseudo-differential operator $\Lambda := \sqrt{-\Delta}$, where the square root is interpreted in the sense of
the Spectral Theorem.
\end{definition}

\begin{remark}
The operator $\Lambda$ is well defined since $-\Delta$ is essentially self-adjoint in $C_c^\infty(\mathbb{R}^N) \subset L^2(\mathbb{R}^N)$~\cite{rs}.
\end{remark}

\begin{remark}
The operator $\Lambda^n$ is a differential, and thus local, operator when $n$ is even; in this case we actually have $\Lambda^n=(-\Delta)^{n/2}$.
If $n$ is odd then $\Lambda^n$ is a nonlocal pseudo-differential operator.
\end{remark}

\begin{proposition}
$\Lambda f = \mathcal{F}^{-1}[2 \pi |\eta| \mathcal{F}(f)]$.
\end{proposition}

\begin{proof}
This is an immediate consequence of the spectral representation of the Laplacian in terms of the Fourier transform:
$$
-\Delta f = \mathcal{F}^{-1}[4 \pi^2 |\eta|^2 \mathcal{F}(f)].
$$
\end{proof}

\begin{definition}\label{funsol2}
We define $G_{n,N} \in \mathcal{S}^*(\mathbb{R}^N)$ in the following way:
\begin{itemize}
\item If $0 < n < N$, it is the unique solution to $\Lambda^n G_{n,N}=\delta_0$ that obeys $G_{n,N}(x) \to 0$ when $|x| \to \infty$.
\item If $n=N$, it is the unique solution to $\Lambda^n G_{n,N}=\delta_0$ in $\text{BMO}(\mathbb{R}^N)$.
\end{itemize}
\end{definition}

\begin{proposition}
The distribution $G_{n,N}$ is given by the exact formulas:
\begin{itemize}
\item If $0 < n < N$,
$$G_{n,N}(x)=\frac{C_{n,N}}{|x|^{N-n}}, \quad C_{n,N}= 2^{-n} \pi^{-N/2}
\frac{\Gamma\left(\frac{N-n}{2}\right)}{\Gamma\left(\frac{n}{2}\right)}.$$
\item If $n = N$,
$$G_{N,N}(x) \equiv G_{N}(x)= C_N \log|x|,$$
\vspace{-0.15cm}
$$C_{N}= \left\{ \begin{array}{cc}
(2-N) (2 \pi)^{N-2} \pi^{2-N/2} \Gamma\left(\frac{N}{2}-1\right), & N \ge 3 \\
-(2\pi)^{-1}, & N=2 \\
\pi^{-1}, & N=1
\end{array}. \right.$$
\end{itemize}
\end{proposition}

\begin{proof}
We use
$$
\mathcal{F}(\Lambda^n G_{n,N})=\mathcal{F}(\delta_0)=1,
$$
to find
$$
\mathcal{F}(G_{n,N})(\xi)=(2 \pi |\xi|)^{-n} \, \forall \, \xi \in \mathbb{R}^N \setminus \{0\}.
$$
When $0<n<N$, $\mathcal{F}(G_{n,N})(\xi)$ is well defined in $L^1(\mathbb{R}^N) + L^\infty(\mathbb{R}^N)$, and therefore it
is well defined as a Schwartz distribution.
Now, an argument akin to that in the proof of Lemma~\ref{liouville} yields
that indeed $\mathcal{F}(G_{n,N})(\xi)=(2 \pi |\xi|)^{-n}$ in $\mathcal{S}^*(\mathbb{R}^N)$.
The statement follows by Fourier inversion.

If $n=N$, then
$$
\mathcal{F}(G_{N})(\xi)=(2 \pi |\xi|)^{-N} \, \forall \, \xi \in \mathbb{R}^N \setminus \{0\}.
$$
Therefore in this case $\mathcal{F}(G_{N}) \not\in L^1_{\text{loc}}(\mathbb{R}^N)$ and it does not even define a singular integral operator.
Consequently our approach in this case will be different; lets start with $\mathbb{R}^2$, in this case
$$
\Lambda^2 G(x)= \delta_0 \Longleftrightarrow -\Delta G(x)= \delta_0,
$$
and so
$$
G(x)=-(2 \pi)^{-1} \log|x|.
$$

Now focus in $N \ge 3$ and compute
\begin{eqnarray} \nonumber
\Lambda^n \log|x| &=& \Lambda^{n-2} (\Lambda^2 \log |x|) \\ \nonumber
&=& \Lambda^{n-2} [(-\Delta) \log |x|] \\ \nonumber
&=& \Lambda^{n-2} \left( \frac{2-n}{|x|^2} \right).
\end{eqnarray}
By means of Fourier transform we find
\begin{eqnarray} \nonumber
\mathcal{F}(\Lambda^n \log|x|)(\xi) &=& (2 \pi |\xi|)^{n-2} \mathcal{F}\left( \frac{2-n}{|x|^2} \right)(\xi) \\ \nonumber
&=& (2 \pi |\xi|)^{n-2}(2-n)\pi^{2-n/2}\Gamma(n/2-1) |\xi|^{-(n-2)} \\ \nonumber
&=:& C_N^{-1};
\end{eqnarray}
note that $C_N$ is always well defined for $N \ge 3$.
Therefore
$$
\mathcal{F}[\Lambda^n (C_N \log|x|)](\xi) = 1 \Longleftrightarrow \Lambda^n (C_N \log|x|) = \delta_0.
$$

It only rests to show that $\Lambda G(x) = \delta_0$ in $\mathbb{R}$. We remind the reader that $G(x) \propto \log|x|$, $d\log|x|/dx=x^{-1}$
and that $x^{-1}$ defines a Schwartz distribution when interpreted as a principal value; in this case
$$
\mathcal{F}\left[ \text{P. V.} \left(\frac{1}{x}\right) \right](\xi)= i \pi \, \text{sgn}(\xi).
$$
Now compute
\begin{eqnarray}\nonumber
i \pi \, \text{sgn}(\xi) &=& \mathcal{F}\left( \frac{d\log|x|}{dx} \right)(\xi), \\ \nonumber
&=& 2 \pi i \, \xi \, \mathcal{F}\left( \log|x| \right)(\xi) \\ \nonumber
&\Rightarrow& \mathcal{F}\left( \log|x| \right)(\xi)= \frac{1}{2|\xi|} \quad \text{if} \quad \xi \neq 0.
\end{eqnarray}
Clearly, $|\xi|^{-1} \not\in S^*(\mathbb{R})$, and therefore $\mathcal{F}\left( \log|x| \right)(\xi)$ has to be interpreted as a renormalization
of $(2|\xi|)^{-1}$.
Now consider
\begin{eqnarray}\nonumber
\mathcal{F}\left( \Lambda^{1/2} \log|x| \right)(\xi) &=& (2 \pi |\xi|)^{1/2} \mathcal{F}\left(\log|x| \right)(\xi) \\ \nonumber
&=& (2 \pi |\xi|)^{1/2} (2|\xi|)^{-1} \\ \nonumber
&=& \sqrt{\frac{\pi}{2}} \, |\xi|^{-1/2},
\end{eqnarray}
if $\xi \neq 0$.
Regularizing the singularity of $\mathcal{F}(\log|x|)(\xi)$ at the origin and letting the regularization parameter go to zero we find
$$
\mathcal{F}\left( \Lambda^{1/2} \log|x| \right)(\xi) = \sqrt{\frac{\pi}{2}} \, |\xi|^{-1/2} \quad \text{in} \quad S^*(\mathbb{R}).
$$
Finally
\begin{eqnarray}\nonumber
\mathcal{F}\left( \Lambda \log|x| \right)(\xi) &=& \mathcal{F}\left[ \Lambda^{1/2} \left( \Lambda^{1/2} \log|x| \right) \right](\xi) \\ \nonumber
&=& (2 \pi |\xi|)^{1/2} \sqrt{\frac{\pi}{2}} \, |\xi|^{-1/2} \\ \nonumber
&=& \pi = \mathcal{F}(\pi \delta_0),
\end{eqnarray}
in $S^*(\mathbb{R})$.
\end{proof}

\begin{proposition}
The distribution $G_{n,N}(x)$ is well defined and, in particular:
\begin{itemize}
\item If $G(x)$ solves $\Lambda^n G = \delta_0$, $0 < n <N$, and $G(x) \to 0$ when $|x| \to \infty$, then $G=G_{n,N}$.
\item If $G(x)$ solves $\Lambda^N G = \delta_0$ and $G(x) \in \text{BMO}(\mathbb{R}^N)$, then $G-G_N$ is constant,
i.~e. $G \equiv G_N$ in $\text{BMO}(\mathbb{R}^N)$.
\end{itemize}
\end{proposition}

\begin{proof}
The existence of this distribution was proven in the previous Proposition and its uniqueness follows analogously as
in the proof of Lemma~\ref{liouville}.
\end{proof}

\begin{theorem}
Let $n \in \mathbb{Z}$, $0 < n \le N$, and
$$
\Lambda^n u = f \quad \text{in} \quad \mathbb{R}^N.
$$
Then $\partial_x^\alpha u = A_{n,N} R^\alpha(f)$ for some constant $A_{n,N}$, where $|\alpha|=n$, the monomial
$\partial^\alpha_x = \partial_{x_{j_1}} \cdots \partial_{x_{j_n}}$, $R^\alpha = R_{x_{j_1}} \cdots R_{x_{j_n}}$
and $R_{x_{1}}, \cdots, R_{x_{n}}$ are the corresponding Riesz transforms in $\mathbb{R}^N$.
\end{theorem}

\begin{proof}
We start with the subcritical case $0 < n < N$:
$$
u(x)= (G_{n,N} \ast f)(x) \equiv \Lambda^{-n} f.
$$
We can write
$$
\Lambda^{n-1} u= \Lambda^{-1} f= C_N \int_{\mathbb{R}^N} \frac{f(y)}{|x-y|^{N-1}} \, dy,
$$
and thus
\begin{eqnarray}\nonumber
\partial_{x_j} \Lambda^{n-1} u &=& C_N (1-N) \,\, \text{P. V.} \int_{\mathbb{R}^N} \frac{x_j - y_j}{|x-y|^{N+1}} \, f(y) \, dy \\ \nonumber
&=& D_N R_{x_j}(f),
\end{eqnarray}
where $D_N \neq 0$ since $N \ge 2$.
Therefore,
\begin{eqnarray}\nonumber
\partial_x^\alpha u &=& \partial_{x_{j_1}} (\partial_{x_{j_2}} \cdots \partial_{x_{j_2}}) \Lambda^{1-n} (\Lambda^{n-1} u) \\ \nonumber
&=& (\partial_{x_{j_2}} \cdots \partial_{x_{j_n}}) \Lambda^{1-n} (\partial_{x_{j_1}} \Lambda^{n-1} u) \\ \nonumber
&=& D_N R_{x_{j_2}} \cdots R_{x_{j_n}} (R_{x_{j_1}} f) \\ \nonumber
&=& D_N R^\alpha(f).
\end{eqnarray}

Now we move to the case $n=N \ge 3$. We know $u=G_N \ast f$ where $G_N = C_N \log|x|$. Then
\begin{eqnarray}\nonumber
\Lambda^{N-1} u &=& \Lambda^{N-3} (-\Delta u) \\ \nonumber
&=& \Lambda^{N-3} [(-\Delta G_N) \ast f],
\end{eqnarray}
where $-\Delta G_N = C_N (2-N) |x|^{-2}$. Therefore
$$
\Lambda^{N-1} u = C_N \int_{\mathbb{R}^N} \frac{f(y)}{|x-y|^{N-1}} \, dy,
$$
where $C_N \neq 0$ and the rest of the proof follows as in the previous case.

When $n=N=2$ we write $u=-(2\pi)^{-1} \log|x| \ast f(x)$ and therefore
\begin{eqnarray}\nonumber
\partial_{x_j} u &=& - \frac{1}{2\pi} \int_{\mathbb{R}^2} \frac{x_j - y_j}{|x-y|^2} \, f(y) \, dy \\ \nonumber
&=& - \frac{1}{2\pi} \int_{\mathbb{R}^2} \frac{y_j}{|y|^2} \, f(x-y) \, dy.
\end{eqnarray}
Finally we have
\begin{eqnarray}\nonumber
\Lambda \partial_{x_j} u(x) &\propto& \text{P. V.} \int_{\mathbb{R}^2} \frac{y_j}{|y|^3} \, f(x-y) \, dy \\ \nonumber
&\propto& R_{x_j} (f)(x),
\end{eqnarray}
and thus
$$
\partial_{x_j} \partial_{x_k} u = (R_{x_k} \Lambda) \partial_{x_j} u \propto R_{x_j} R_{x_k} u.
$$

The case $n=N=1$ comes from the fact that
$$
u(x)= \frac{1}{\pi} \int_{\mathbb{R}} \log|x-y| \, f(y) \, dy,
$$
and the fact that
$$
u'(x)= \frac{1}{\pi} \,\, \text{P. V.} \int_{\mathbb{R}} \frac{x_j - y_j}{|x-y|^2} \, f(y) \, dy,
$$
which is the Hilbert transform of $f$.
\end{proof}

\begin{corollary}\label{exunnl}
The linear equation
$$
\Lambda^n u = \lambda f, \qquad x \in \mathbb{R}^N,
$$
has a unique solution in the following cases:
\begin{itemize}
\item[(a)] $f \in L^p(\mathbb{R}^N), \quad 1 < p < \frac{N}{n}, \quad n < N$,
\item[(b)] $f \in L^1(\mathbb{R}^N), \quad n < N$,
\item[(c)] $f \in \mathcal{H}^1(\mathbb{R}^N), \quad n < N$,
\item[(d)] $f \in \mathcal{H}^1(\mathbb{R}^N), \quad n = N$.
\end{itemize}
Then, respectively
\begin{itemize}
\item[(a)] $u \in \dot{W}^{n-\epsilon,N p/(N-\epsilon p)}(\mathbb{R}^N) \, \forall \, 0 \le \epsilon \le n$,
\item[(b)] $u \in \dot{W}^{n-\epsilon,N/(N-\epsilon)}(\mathbb{R}^N) \, \forall \, 0<\epsilon \le n$,
\item[(c)] $u \in \dot{W}^{n-\epsilon,N/(N-\epsilon)}(\mathbb{R}^N) \, \forall \, 0 \le \epsilon \le n$,
\item[(d)] $u \in \dot{W}^{N-\epsilon,N/(N-\epsilon)}(\mathbb{R}^N) \, \forall \, 0 \le \epsilon \le N$.
\end{itemize}
Moreover, in case (b), $D^{n} u \in L^{1,\infty}(\mathbb{R}^N)$ and, in cases (c) and (d), $D^{n} u \in \mathcal{H}^{1}(\mathbb{R}^N)$.
\end{corollary}

Now we state the main result of this section. Of course, $N, k \in \mathbb{N}$ always and we also assume $N > 2k$.

\begin{theorem}\label{exnlnl}
Equation~\eqref{rkhessiannl} has at least one weak solution in the following cases:
\begin{itemize}
\item[(a)] $f \in L^p(\mathbb{R}^N), \quad 1 < p < \frac{N}{2k}, \quad n = 2 + N (k-1)/(pk) \in \mathbb{N}$,
\item[(b)] $f \in L^1(\mathbb{R}^N), \quad n=2 + N (k-1)/k \in \mathbb{N}$,
\item[(c)] $f \in \mathcal{H}^1(\mathbb{R}^N), \quad n=2 + N (k-1)/k \in \mathbb{N}$,
\end{itemize}
provided $|\lambda|$ is small enough.
Then, respectively
\begin{itemize}
\item[(a)] $u \in \dot{W}^{n-\epsilon,N p/(N-\epsilon p)}(\mathbb{R}^N) \, \forall \, 0 \le \epsilon \le n$,
\item[(b)] $u \in \dot{W}^{n-\epsilon,N/(N-\epsilon)}(\mathbb{R}^N) \, \forall \, 0<\epsilon \le n$,
\item[(c)] $u \in \dot{W}^{n-\epsilon,N/(N-\epsilon)}(\mathbb{R}^N) \, \forall \, 0 \le \epsilon \le n$.
\end{itemize}
Moreover, in case (b), $D^{n} u \in L^{1,\infty}(\mathbb{R}^N)$ and, in case (c), $D^{n} u \in \mathcal{H}^{1}(\mathbb{R}^N)$.
Also, for a smaller enough $|\lambda|$, the solution is locally unique in every case.
\end{theorem}

\begin{proof}
The proof follows as a consequence of corollary~\ref{exunnl} and going through the same arguments as in Section~\ref{existence}
and~\ref{locuniq}.
\end{proof}

\begin{remark}
The case $N=2k$ was already examined in Theorem~\ref{exh1}.
\end{remark}

\section{Further results}
\label{further}

Our previous results imply the weak continuity of the branch of solutions that departs from $u=0$ and $\lambda=0$ under certain conditions.

\begin{theorem}
\label{wcbranch}
Let
\begin{eqnarray}\nonumber
\Phi: D(\Phi) \subset \mathcal{B} &\longrightarrow& \mathcal{B} \\ \nonumber
v &\longmapsto& u(v),
\end{eqnarray}
where $u$ is the unique solution to
$$
\Lambda^n u = S_k[-u] + \lambda f, \qquad x \in \mathbb{R}^N,
$$
$v= \Lambda^{-n} f$, $\mathcal{B} = \dot{W}^{n,p}(\mathbb{R}^N)$,
$f \in L^p(\mathbb{R}^N)$ and the rest of hypotheses as in Theorem~\ref{exnlnl}.
Then $\Phi$ is weakly continuous, i.~e. $\forall \, \{v_j\}_j \subset D(\Phi)$ such that
$$
v_j \rightharpoonup v \text{ weakly in } \mathcal{B},
$$
it holds that
$$
\Phi(v_j) \rightharpoonup \Phi(v) \text{ weakly in } \mathcal{B}.
$$
\end{theorem}

\begin{proof}
Take $D(\Phi)$ to be the ball in $\mathcal{B}$ used in Theorem~\ref{exulp}.
Then we know $\Phi$ is well defined and moreover $\Phi: D(\Phi) \longrightarrow D(\Phi)$.
We rewrite our equation
$$
u_j = \Lambda^{-n} \left( S_k[-u_j] \right) + \lambda v_j;
$$
we know that for every $v_j \in D(\Phi)$ there exist a unique solution $u_j \in D(\Phi)$.
Now take the limit $j \to \infty$ and we conclude by weak continuity of $S_k[\cdot]$ in $L^p(\mathbb{R}^N)$,
see the proof of Theorem~\ref{exlp}.
\end{proof}

We also have a comparatively weaker result for summable data.

\begin{theorem}
Let
\begin{eqnarray}\nonumber
\Phi: D(\Phi) \subset \mathcal{B} &\longrightarrow& \mathcal{B} \\ \nonumber
v &\longmapsto& u(v),
\end{eqnarray}
where $u$ is the unique solution to
$$
\Lambda^n u = S_k[-u] + \lambda f, \qquad x \in \mathbb{R}^N,
$$
$v= \Lambda^{-n} f$, $\mathcal{B} = \dot{W}^{n-1,N/(N-1)}(\mathbb{R}^N)$,
$f \in L^1(\mathbb{R}^N)$ and the rest of assumptions as in Theorem~\ref{exnlnl}.
Then $\Phi$ is weakly continuous, i.~e. $\forall \, \{v_j\}_j \subset D(\Phi)$ such that
$$
v_j \rightharpoonup v \text{ weakly in } \mathcal{B},
$$
it holds that
$$
\Phi(v_j) \rightharpoonup \Phi(v) \text{ weakly in } \mathcal{B}.
$$
\end{theorem}

\begin{proof}
The proof follows as the proof of Theorem~\ref{wcbranch} combined with the arguments regarding weak continuity in the proof of Theorem~\ref{exl1}.
\end{proof}

In the following we will improve our regularity results from sections~\ref{ltheory} and~\ref{nproblems}
and guarantee that the solution of the critical case obeys the boundary conditions.

\begin{theorem}\label{higherreg}
Let $f \in \mathcal{H}^1(\mathbb{R}^N)$, then $\Lambda^{-N}f \in C_0(\mathbb{R}^N)$ and
$$
\Lambda^{-N}:\mathcal{H}^1(\mathbb{R}^N) \longrightarrow C_0(\mathbb{R}^N)
$$
is bounded.
\end{theorem}

\begin{proof}
We already know from Proposition~\ref{exun} that $\|\Lambda^{-N} f\|_{L^\infty(\mathbb{R}^N)} \ll \|f\|_{\mathcal{H}^1(\mathbb{R}^N)}$.
Now let $a$ be a $L^\infty$-atom for $\mathcal{H}^1(\mathbb{R}^N)$, i.~e. $a \in \mathcal{H}^1(\mathbb{R}^N)$ and
\begin{itemize}
\item There exists a $\mathbb{R}^N-$cube $Q \subset \mathbb{R}^N$, that is $Q=c(Q)+ \ell(Q) Q_0$ with $c(Q) \in \mathbb{R}^N$,
$Q_0 = [-1/2,1/2]^N$ and $\ell(Q) > 0$,
such that $a$ is supported on $Q$,
\item $\|a\|_{L^\infty(Q)} \le |Q|^{-1}$,
\item $\int_Q a \, dx =0$.
\end{itemize}
We start proving that $\Lambda^{-N}a \in C_0(\mathbb{R}^N)$. Let $x \in \mathbb{R}^N$ be such that $|x - c(Q)| \ge 2 \sqrt{N} \ell$. Then
\begin{eqnarray}\nonumber
\Lambda^{-N} a(x) &=& C_N \int_Q \log |x-y| a(y) dy \\ \nonumber
&=& C_N \int_Q \left[ \log |x-y| - \log |x-c(Q)| \right] a(y) dy,
\end{eqnarray}
after the use of the first and third defining properties of $a$ in the first and second equalities respectively.
If $y \in Q$, then
\begin{eqnarray}\nonumber
|x-y| &=& |[x-c(Q)]-[y-c(Q)]| \\ \nonumber
&\ge& |x-c(Q)| - |y-c(Q)| \\ \nonumber
&\ge& \frac34 |x - c(Q)| >0,
\end{eqnarray}
where we have used
$$
|y-c(Q)| \le \frac12 \sqrt{N} \ell \le \frac14 |x - c(Q)|.
$$
The same reasoning leads to conclude
$$
\frac34 \le \frac{|x - y|}{|x - c(Q)|} \le \frac54,
$$
and then
$$
\frac{|x - y|}{|x - c(Q)|} = 1 + t, \qquad |t| \le \frac14.
$$
The triangle inequality again gives
$$
\left| \, |x-y| - |x-c(Q)| \, \right| \le |y - c(Q)|,
$$
which implies
$$
|t| \le \frac{|y-c(Q)|}{|x-c(Q)|},
$$
and then
$$
\left| \log \left[ \frac{|x-y|}{|x-c(Q)|} \right] \right| \le C \frac{|y-c(Q)|}{|x-c(Q)|}.
$$
Therefore
\begin{eqnarray}\nonumber
\left| \Lambda^{-N} a(x) \right| &\ll& \int_Q \frac{|y-c(Q)|}{|x-c(Q)|} |a(y)| dy \\ \nonumber
&\ll& \frac{1}{|x-c(Q)|} \int_Q |y-c(Q)| \, |Q|^{-1} dy \\ \nonumber
&\ll& \frac{\ell(Q)}{|x-c(Q)|}.
\end{eqnarray}
Since this last estimate holds for $|x-c(Q)| \ge 2 \sqrt{N} \ell$ and
$$
\left| \Lambda^{-N} a(x) \right| \ll \|a\|_{\mathcal{H}^1(\mathbb{R}^N)} \ll 1,
$$
it follows that
$$
\left| \Lambda^{-N} a(x) \right| \ll \frac{\ell(Q)}{\ell(Q) + |x-c(Q)|} \, \forall \, x \in \mathbb{R}^N,
$$
which proves the decay in the limit $|x| \to \infty$.

To prove continuity of $\Lambda^{-N} a(x)$ choose $x,h \in \mathbb{R}^N$ to find
\begin{eqnarray}\nonumber
& & \left| \Lambda^{-N} a(x+h) - \Lambda^{-N} a(x) \right| \\ \nonumber
&=& C_N \left| \int_Q ( \log|x+h-y| - \log|x-y| ) a(y) \, dy \right| \\ \nonumber
&\ll& \| a \|_{L^\infty} \int_Q \left| \, \log|x+h-y| - \log|x-y| \, \right| dy \\ \nonumber
&=& \int_{Q_0} \left| \, \log\left|\frac{x-c+h}{\ell}-z\right| - \log\left|\frac{x-c}{\ell}-z\right| \, \right| dz \\ \nonumber
&=:& F\left(\frac{x-c}{\ell},\frac{h}{\ell}\right),
\end{eqnarray}
where we have used the change of variables $y = \ell z + c$ in the previous to last step.
It is enough to prove continuity of $F$ and we may assume $0 < |h| \le \frac14$.
Since $Q_0 \subset B_{\sqrt{N}/2}(0) =: B$, we have
\begin{eqnarray}\nonumber
F(x,h) &\le& \int_B \left| \, \log|x+h-y| - \log|x-y| \, \right| dy \\ \nonumber
&=& |h|^N \int_{|h|^{-1} B} \left| \, \log|x' + h' - u| - \log|x' - u| \, \right| du,
\end{eqnarray}
after the change of variables $y=|h|u$, and where $x'=x/|h|$ and $h'=h/|h| \in \mathbb{S}^{N-1}$.
If $|x| \ge \sqrt{N}$ then $|x'-u| \ge |x'|-|u| \ge \sqrt{N}/(2|h|)$ for $u \in |h|^{-1} B$.
Therefore
\begin{eqnarray}\nonumber
\log|x' + h' - u| - \log|x' - u| &=& \log \left| \frac{x' - u}{|x' - u|} + \frac{h'}{|x' - u|} \right| \\ \nonumber
&=& O\left(\frac{|h'|}{|x' - u|}\right) = O\left(|h|\right).
\end{eqnarray}
Then
$$
F(x,h) \ll |h|^N \int_{|h|^{-1} B} |h| \, du \ll |h|,
$$
which proves continuity in this case.

If $|x| \le \sqrt{N}$ then $B-x \subset B_{3\sqrt{N}/2}(0)=3B$ and
\begin{eqnarray}\nonumber
F(x,h) &\le& |h|^N \int_{3B |h|^{-1}} \left| \, \log|z+h'| - \log|z| \, \right| dz \\ \nonumber
&=& |h|^N \int_{\{3B |h|^{-1}\}\cap\{|z| \le 2\}} \left| \, \log|z+h'| - \log|z| \, \right| dz \\ \nonumber
& & + \, |h|^N \int_{\{3B |h|^{-1}\}\cap\{|z| \ge 2\}} \left| \, \log|z+h'| - \log|z| \, \right| dz \\ \nonumber
&=:& I_1 + I_2.
\end{eqnarray}
after the change of variables $y = x + |h|z$ in the first step.
The first term can be estimated as follows
$$
I_1 \le |h|^N \int_{|z| \le 2} \left| \, \log|z+h'| - \log|z| \, \right| dz \ll |h|^N,
$$
since the integral can be bounded by a constant independent of $h'$.
For the second term we find
\begin{eqnarray}\nonumber
I_2 &=& |h|^N \int_{2 \le |z| \le 3\sqrt{N}/(2|h|)} \left| \, \log\left|\frac{z}{|z|}+\frac{h'}{|z|}\right| \, \right| dz \\ \nonumber
&\ll& \left\{ \begin{array}{cc}
|h| \, \log \left(\frac{1}{|h|}\right), & n=1 \\
|h| \, & n>1
\end{array} \right. ,
\end{eqnarray}
because the integrand is $O\left(|z|^{-1}\right)$. Summing up:
\begin{eqnarray}\nonumber
& & \left| \Lambda^{-N}a(x+h) - \Lambda^{-N}a(x) \right| \\ \nonumber
&\ll& \left\{ \begin{array}{cc}
\min\left\{1,|h|\left[1+ \log \left(\frac{1}{|h|}\right) \right]\right\}, & n=1 \\
\min\left\{1,|h|\right\}, & n>1
\end{array} \right.
\, \forall \, x,h \in \mathbb{R}^N,
\end{eqnarray}
such that $0 < |h| \le 1/4$.

Therefore
$$
\Lambda^{-N}:\mathcal{H}^1_{\text{at}}(\mathbb{R}^N) \longrightarrow C_0(\mathbb{R}^N),
$$
where $\mathcal{H}^1_{\text{at}}(\mathbb{R}^N)$ is the set of all finite linear combinations of
$L^\infty(\mathbb{R}^N)$-atoms for $\mathcal{H}^1(\mathbb{R}^N)$.
Since $\mathcal{H}^1_{\text{at}}(\mathbb{R}^N)$ is dense in $\mathcal{H}^1(\mathbb{R}^N)$ for $f \in \mathcal{H}^1(\mathbb{R}^N)$
there exists $f_j \in \mathcal{H}^1_{\text{at}}(\mathbb{R}^N)$ such that $f_j \rightarrow f$ in $\mathcal{H}^1(\mathbb{R}^N)$,
and therefore $\Lambda^{-N} f_j \rightarrow \Lambda^{-N} f$ in $L^\infty(\mathbb{R}^N)$.
Uniform convergence guarantees that $\Lambda^{-N} f$ is not only bounded but also continuous.

Now we prove that $\Lambda^{-N} f(x) \to 0$ when $|x| \to \infty$.
Uniform convergence of $\Lambda^{-N} f_j(x)$ to $\Lambda^{-N} f(x)$ implies that there exists a $J \in \mathbb{N}$ such that
for $j \ge J$ it holds that $|\Lambda^{-N} f(x) - \Lambda^{-N} f_j(x)| \le \epsilon/2 \,\, \forall \, x \in \mathbb{R}^N$.
Now fix such a $j \ge J$. Since $\Lambda^{-N} f_j(x) \to 0$ when $|x| \to \infty$,
then there exist $0 < R < \infty$ such that for $|x| \ge R$ it holds that $|\Lambda^{-N} f_j(x)| \le \epsilon/2$.
In consequence for $|x| \ge R$,
\begin{eqnarray}\nonumber
|\Lambda^{-N} f(x)| &=& |\Lambda^{-N} f(x) - \Lambda^{-N} f_j(x) + \Lambda^{-N} f_j(x)| \\ \nonumber
&\le& |\Lambda^{-N} f(x) - \Lambda^{-N} f_j(x)| +|\Lambda^{-N} f_j(x)| \\ \nonumber
&\le& \epsilon.
\end{eqnarray}
\end{proof}

\begin{corollary}\label{corc0}
The solution whose existence was proven in Theorem~\ref{exh1} actually belongs to $C_0(\mathbb{R}^N)$ in the critical case $2m=N=2k$.
\end{corollary}

\vskip5mm
\noindent
{\footnotesize
Pedro Balodis\par\noindent
Departamento de Matem\'aticas\par\noindent
Universidad Aut\'onoma de Madrid\par\noindent
{\tt pedro.balodis@uam.es}\par\vskip1mm\noindent
\& \par\vskip1mm\noindent
Carlos Escudero\par\noindent
Departamento de Matem\'aticas\par\noindent
Universidad Aut\'onoma de Madrid\par\noindent
{\tt carlos.escudero@uam.es}\par\vskip1mm\noindent
}
\end{document}